\documentclass[12pt]{article}

\usepackage[utf8]{inputenc}
\usepackage{geometry}
\usepackage{amsmath,amsfonts,amssymb,amsthm,mathtools}
\usepackage{graphicx}
\usepackage{subcaption}   

\usepackage{booktabs}
\usepackage{multirow}
\usepackage{array}
\usepackage{float}
\usepackage{placeins}

\usepackage{xcolor,color}
\usepackage[colorlinks=true,citecolor=blue,linkcolor=blue]{hyperref}

\usepackage{authblk}
\usepackage{algorithm}
\usepackage{dsfont}
\usepackage{mathrsfs}
\usepackage{upgreek}
\usepackage{pgfplots}
\pgfplotsset{compat=1.7}
\usetikzlibrary{intersections}

\usepackage{cite}
\usepackage[numbers]{natbib}

\usepackage{fancyhdr}
\usepackage{etoolbox}
\usepackage{accents}

\newtheorem{thm}{Theorem}[section]
\newtheorem{lemma}{Lemma}[section]

\numberwithin{equation}{section}
\numberwithin{figure}{section}
\numberwithin{table}{section}

\begin{document}
\title{{\em Error Analysis of a Conforming FEM for  Multidimensional Fragmentation Equations}}

\author{
    Arushi\thanks{Department of Mathematics \& Computing, 
    Dr B R Ambedkar National Institute of Technology Jalandhar, Punjab, India. 
    Email:  arushi.mc.25@nitj.ac.in}
    \;\; and \;\;
    Naresh Kumar\thanks{Department of Mathematics \& Computing, 
    Dr B R Ambedkar National Institute of Technology Jalandhar, Punjab, India. 
    Corresponding author: nareshk@nitj.ac.in}
}
\date{}
\maketitle

\begin{abstract}
In this work, we develop and analyze a higher-order finite element method for the multidimensional fragmentation equation. To the best of our knowledge, this is the first study to establish a rigorous, conforming finite element framework for high-order spatial approximation of multidimensional fragmentation models. The scheme is formulated in a variational setting, and its stability and convergence properties are derived through a detailed mathematical analysis. In particular, the $L^2$ projection operator is used to obtain optimal-order spatial error estimates under suitable regularity assumptions on the exact solution. For temporal discretization, a second-order backward differentiation formula (BDF2) is adopted, yielding a fully discrete scheme that achieves second-order convergence in time. The theoretical analysis establishes $ L^2$-optimal convergence rates of ${\cal O}(h^{r+1})$ in space, together with second-order accuracy in time. The theoretical findings are validated through a series of numerical experiments in two and three space dimensions. The computational results confirm the predicted error estimates and demonstrate the robustness of the proposed method for various choices of fragmentation kernels and selection functions.
\end{abstract}

\noindent\textbf{Keywords:} 
Multidimensional fragmentation;  Integro-differential models; 
Higher-order finite elements;  BDF2 scheme; Error analysis.

{\em AMS Subject Classifications(2000)}. 65N15, 65N30

\section{Introduction}\label{sec:1}

\subsection{Modelling background}
The Population Balance model (PBM) is an integro-partial differential equation that provides a unified framework for understanding and predicting the time evolution of dispersed particle systems within a continuous medium. Because particles in real-world scenarios are generally polydisperse (i.e., they vary in size rather than being uniform), the PBM helps predict how the Particle Size Distribution (PSD) changes due to ``sink" and ``source" events over time. The PBM treats fragmentation as the death of large-sized particles and the birth of the smaller particles \cite{ Nicmanis1998, Rigopoulos2003, Singh2022}. Fragmentation behavior can be influenced by external conditions, such as temperature, external forces, and the medium, as well as by internal conditions, including object size, internal porosity, object shape, and fractal dimension. How a particle moves and reacts to the liquid or gas around it depends largely on its size. The PSD is simply a way to describe the number of particles of each size that exist in our system. This crucial power of PBM makes it indispensable across various sectors, facilitating the study of phenomena like aerosols, crystallization, pharmaceutical sciences, food processing, polymer degradation, and granular flows \cite{ Seinfeld1986, Saha2008, fu2020fragmentation, ramkrishna2014population}. As one-dimensional PBMs are insufficient to capture the complete dynamics of the granulation process \cite{iveson2002limitations}, these limitations of 1D PBMs necessitate the use of multidimensional PBMs.

\subsection{Problem Description}
From a mathematical standpoint, the linear fragmentation population balance equation (PBE) 
describes the temporal evolution of a particle system undergoing pure breakage. 
The model takes the form of a multidimensional integro-partial differential equation.
Let $\mathbf{x} := (x_1,x_2) \in \mathbb{R}_+^2 := (0,\infty)\times(0,\infty)$ 
denote the vector of particle properties, and let 
$u(\mathbf{x},t) \ge 0$ represent the number density of particles 
with characteristics $\mathbf{x}$ at time $t>0$. 
The dynamics of the system are governed by the linear fragmentation equation
\begin{equation}
\partial_t u(\mathbf{x},t)
=
\underbrace{
\int_{\mathbf{x}}^{\infty}
{\beta}(\mathbf{x}\mid \mathbf{y})\, {\Gamma}(\mathbf{y})\, u(\mathbf{y},t)
\, d\mathbf{y}
}_{\text{gain due to fragmentation}}
-
\underbrace{
{\Gamma}(\mathbf{x})\,u(\mathbf{x},t)
}_{\text{loss due to breakage}},
\qquad 
\mathbf{x}\in\mathbb{R}_+^2,\; t>0,
\label{eq:model}
\end{equation}
subject to the initial condition
\begin{equation}
u(\mathbf{x},0) = u_{in}(\mathbf{x}) \ge 0.
\label{eq:model_ic}
\end{equation}
For brevity, we employ the notation
\[
\int_{\mathbf{x}}^{\infty} (\cdot)\, d\mathbf{y}
:=
\int_{x_1}^{\infty} \int_{x_2}^{\infty} (\cdot)\, dy_2\,dy_1.
\]
Equation \eqref{eq:model} naturally separates into two competing mechanisms: 
a \emph{birth (gain)} term representing the creation of smaller fragments 
from larger parent particles, and a \emph{death (loss)} term describing the removal of particles that undergo breakage.

The model ingredients are interpreted as follows:
\begin{itemize}
    \item ${\Gamma}(\mathbf{x}) \ge 0$ is the \emph{selection (breakage) rate}, 
    quantifying how frequently particles of size $\mathbf{x}$ fragment.

    \item ${\beta}(\mathbf{x}\mid\mathbf{y}) \ge 0$ is the \emph{fragmentation kernel}, 
    which characterizes the distribution of daughter particles of size 
    $\mathbf{x}$ produced from a parent particle of size $\mathbf{y}$.
\end{itemize}
To ensure physical consistency, the kernel satisfies the following structural properties:
\begin{enumerate}
    \item \textbf{No oversize fragments.}   Fragments cannot exceed the size of their parent: $ {\beta}(\mathbf{x}\mid\mathbf{y}) = 0, \text{whenever } x_1 > y_1 \ \text{or}\ x_2 > y_2.$

    \item \textbf{Multiplicity of fragments.}  
    Each breakage event generates more than one daughter particle:
    \begin{equation}
    \int_{(0,y_1)\times(0,y_2)}
    {\beta}(\mathbf{x}\mid\mathbf{y})\, d\mathbf{x}
    =
    \nu(\mathbf{y}) > 1,
    \qquad \forall\, \mathbf{y}\in\mathbb{R}_+^2.
    \label{eq:multiplicity}
    \end{equation}
\item \textbf{Conservation of additive quantities.}  
    If $x_1$ and $x_2$ represent additive properties (such as mass or volume components), 
    the total quantity is preserved during fragmentation:
    \begin{equation}
    \int_{(0,y_1)\times(0,y_2)}
    (x_1+x_2)\, {\beta}(\mathbf{x}\mid\mathbf{y})\, d\mathbf{x}
    =
    y_1 + y_2.
    \label{eq:mass}
    \end{equation}
\end{enumerate}
Together, these assumptions ensure that the model captures the essential 
features of a pure breakage process: particles split into smaller fragments, 
their total number increases, and conserved physical quantities remain invariant.

Apart from the number density function itself, some integral properties, such as the moments, are helpful to understand a distribution. For a distribution function $u(x_1, x_2, t)$, the $(r, s)$th moment is defined as
\begin{equation}
M_{r,s}(t) = \int_{0}^{\infty} \int_{0}^{\infty} x_1^r x_2^s u(x_1, x_2, t) \, dx_2 \, dx_1 \quad \text{where} \quad r, s \geq 0.
\end{equation}
Some initial instances of moments are extremely intriguing. The zeroth moment ($r = 0, s = 0$) represents the total number of particles in the system, and the total mass of the system is denoted by the sum of first moments ($M_{1,0} + M_{0,1}$), this property remains invariant throughout the entire process. Furthermore, it is easy to obtain
\begin{equation}
\frac{dM_{0,0}}{dt} = \int_{0}^{\infty} \int_{0}^{\infty} (\nu(y_1, y_2) - 1) {\Gamma}(y_1, y_2) u(y_1, y_2, t) \, dy_2 \, dy_1,
\end{equation}
where $\nu(y_1, y_2)$ is the number of fragments produced from a parent of size $(y_1, y_2)$.
\begin{equation}
\frac{d}{dt}(M_{1,0} + M_{0,1}) = 0.
\end{equation}
As expected due to the breakage event, clearly, $\frac{dM_{0,0}}{dt} > 0$, this follows from the fact that ${\Gamma}(y_1, y_2)$, $u(y_1, y_2, t)$ and $(\nu(y_1, y_2) - 1)$ are all positive.

\subsection{The state of the art}
In the literature, various numerical methods are available to accurately predict moments in population balance equations. These include the sectional methods such as the cell average technique \cite{kumar2008cell} and the fixed pivot technique \cite{leong2023comparative}, the method of moments \cite{falola2013extended}, Monte Carlo simulations \cite{bhoi2019sonofragmentation}, and finite volume methods \cite{Singh2022, singh2022finite, Kumar2011}. The main drawback of the cell-average technique is its high computational cost, as it requires recomputing the birth term after redistributing particle properties to neighboring nodes \cite{singh2022new}. In FPT, artificial generation of fines occurs \cite{singh2022challenges}. Moment methods are useful only for low-order moments, but they cannot reconstruct the complete distribution without additional assumptions or closure relations. Finite volume methods are widely used for their simplicity and efficiency, but FVMs are only preferred when a small number of moments are needed, i.e., up to two \cite{leong2023comparative}. Also, FVMs face limitations in multidimensional settings \cite{Saha2008, Singh2022}. Generally, FVM can give oscillatory solutions whenever large grid sizes are used. Monte Carlo methods simulate the stochastic evolution of individual particles and are well-suited for handling multidimensional distributions; however, they often exhibit low accuracy and produce noisy results. Also, the computational cost of the Monte Carlo method is extremely high \ \cite{fan2017direct}. Existing FEM approaches often relied on coarse numerical approximations when evaluating integral terms in these types of models. The finite element method for the one-dimensional fragmentation equation provides a systematic framework for tracking the evolution of particle size \cite{sangwan2026priori}, but this approach is still insufficient when capturing the multidimensional dynamics of particulate systems is required. The least-squares method can also be used to solve population balance equations, but when a Lagrange multiplier is introduced to enforce mass conservation, positive definiteness can be lost \cite{zhu2010mass}. The discontinuous Galerkin method preserves positive definiteness, but its computational cost is relatively high \cite{xu2026conservative}.

\subsection{Contributions of the Present Work}
Unlike many existing techniques that lack a continuous solution across boundaries (cell or control-volume), the finite element method provides a globally continuous approximation by employing $C^0$ finite element spaces. The finite element method is preferred in this work, utilizing its mathematically rigorous, weak (variational) formulation to address these challenges effectively. In this study, we focus on multidimensional fragmentation. Two-dimensional linear fragmentation equations are mathematical models that help to describe the breakup of a population of objects, characterized by two internal properties, into smaller fragments over time. Sometimes, a single characteristic may be insufficient to fully capture the dynamics of the fragmentation process; in such cases, multidimensional equations are vital. Therefore, 2D fragmentation enables accurate predictions of particle size distributions and moment evolution. We developed a higher-order FEM framework with the following key features:

\vspace{-0.3cm}

\begin{enumerate}

\item [(i)] We have used high-order Lagrange elements as a basis function, which can accurately capture complex particle properties. Also, we have handled the loss (death) and gain (birth) terms with consistent numerical integration within the variational formulation (with the restriction that fragments cannot exceed the size of their parent), avoiding the use of improvised methods. 

\item[(ii)] We have derived semi-discrete and fully discrete formulations (with BDF2). These formulations establish the theoretical convergence bounds for the proposed scheme.

\item[(iii)] Our formulation ensures that physical properties, such as mass and moments, are strictly conserved. By this approach, higher-order moment accuracy is maintained and also improved by refining the mesh.

\item [(iv)] For the reliability of the framework, we present the comparison of the numerical solution obtained against the exact solution available or the exact moments available for both 2D and 3D examples.
\end{enumerate}

\noindent{\bf Organization of the paper}: This is how this paper is organized. We truncate the domain for the numerical approximation, followed by the derivation of the weak form for the FEM framework in Section~\ref{sec:2}. We explore the mass conservation property and derive the stability result and error estimate for the semi-discrete finite element method in Section~\ref{sec:3}. In Section~\ref{sec:4}, the error estimate and stability result for the fully-discrete algorithm are presented. Section~\ref{sec:5} is dedicated to the numerical validation of the proposed scheme for both 2D and 3D linear fragmentation equations.  Finally, in Section \ref{sec:6}, we conclude the article.

\vspace{-0.8cm}

\section{Finite Element Approximation} \label{sec:2}
\vspace{-0.26cm}

In this section, we construct a Galerkin finite element approximation of the truncated fragmentation model posed on a bounded computational domain.

\noindent \textbf {Continuous weak formulation:}  The fragmentation model is naturally posed on the unbounded domain 
$\mathbb{R}_+^2 = (0,\infty)\times(0,\infty)$. 
However, for numerical approximation, computations must be carried out 
on a bounded region. Therefore, we introduce the truncated domain
$
 \mathcal{D} := (0,x_{1\max}] \times (0,x_{2\max}],
\; 
0 < x_{1\max},\, x_{2\max} < \infty,
$
and restrict the analysis to a finite time interval $[0,T]$ with $T<\infty$.

\noindent On $ \mathcal{D} \times (0,T]$, the truncated fragmentation problem reads:
Find $u(\mathbf{x},t)$ such that
\begin{equation}
\partial_t u(\mathbf{x},t)
+ {\Gamma}(\mathbf{x})\,u(\mathbf{x},t)
=
\int_{\mathbf{x}}^{\mathbf{x}_{\max}}
{\beta}(\mathbf{x}\mid\mathbf{y})\,
{\Gamma}(\mathbf{y})\,
u(\mathbf{y},t)\,
d\mathbf{y},
\qquad 
\mathbf{x}\in \mathcal{D},\; t>0,
\label{eq:truncated_model}
\end{equation}
with initial condition
\begin{equation}
u(\mathbf{x},0)=u_{in}(\mathbf{x}),
\qquad \mathbf{x}\in \mathcal{D},
\label{eq:truncated_ic}
\end{equation}
where
\[
\int_{\mathbf{x}}^{\mathbf{x}_{\max}} (\cdot)\, d\mathbf{y}
:=
\int_{x_1}^{x_{1\max}}
\int_{x_2}^{x_{2\max}}
(\cdot)\, dy_2\,dy_1.
\]
To ensure that problem \eqref{eq:truncated_model}--\eqref{eq:truncated_ic} 
is mathematically meaningful and suitable for numerical analysis, 
we impose the following regularity assumptions: The selection rate ${\Gamma}(\mathbf{x})$ is assumed to be essentially bounded and nonnegative on $ \mathcal{D}$, that is, $a \in L^\infty( \mathcal{D})$ with $0 \le {\Gamma}(\mathbf{x}) \le a_0 < \infty$ for almost every $\mathbf{x} \in  \mathcal{D}$. 
Moreover, the weighted fragmentation kernel is essentially bounded on $ \mathcal{D} \times  \mathcal{D}$, namely ${\beta}(\mathbf{x}\mid\mathbf{y})\,{\Gamma}(\mathbf{y}) \in L^\infty( \mathcal{D} \times  \mathcal{D})$ and $|{\beta}(\mathbf{x}\mid\mathbf{y})\,{\Gamma}(\mathbf{y})| \le b_0$ for some constant $b_0 > 0$ and the initial number density satisfies $u_{in} \in L^2( \mathcal{D})$ and $u_{in} \ge 0$ almost everywhere in $ \mathcal{D}$. Under these conditions, the integral operator on the right-hand side 
of \eqref{eq:truncated_model} is well-defined and bounded in $L^2( \mathcal{D})$,  providing a stable functional framework for subsequent discretization of finite elements.

\noindent\textbf{Weak formulation:}  
Find $u:(0,T]\to V:=H^1_0( \mathcal{D})$ such that
\begin{equation}
(\partial_t u,\varphi)
+ \mathcal{A}(u,\varphi)
=
\mathcal{B}(u,\varphi),
\qquad 
\forall \varphi\in V,\; t>0,
\label{eq:weak_form}
\end{equation}
with $u(0)=u_{in}$, where
\begin{align*}
(\partial_t u,\varphi)
&:= \int_{ \mathcal{D}} \partial_t u(\mathbf{x},t)\,\varphi(\mathbf{x})\, d\mathbf{x}, \;\;
\mathcal{A}(u,\varphi)
:= \int_{ \mathcal{D}} {\Gamma}(\mathbf{x})\,u(\mathbf{x},t)\,\varphi(\mathbf{x})\, d\mathbf{x}, \\
\mathcal{B}(u,\varphi)
&:= \int_{ \mathcal{D}}
\varphi(\mathbf{x})
\left(
\int_{\mathbf{x}}^{\mathbf{x}_{\max}}
{\beta}(\mathbf{x}\mid\mathbf{y})\,
{\Gamma}(\mathbf{y})\,
u(\mathbf{y},t)\,
d\mathbf{y}
\right)
d\mathbf{x}.
\end{align*}

\noindent\textbf{Computational domain and mesh:}
Let  $ \mathcal{D} := (0,x_{1}^{\star}] \times (0,x_{2}^{\star}]
$ denote the truncated configuration domain in the particle property space. 
We consider a conforming and shape-regular triangulation 
$
\mathscr{T}_h = \{{\cal K}\}
$ of $\mathcal{D}$, consisting of non-overlapping closed triangular elements ${\cal K}$, such that
$
\overline{\mathcal{D}} = \bigcup_{{\cal K} \in \mathscr{T}_h} \overline{{\cal K}}.
$
For any two distinct elements ${\cal K}_1, {\cal K}_2 \in \mathscr{T}_h$, their intersection is either empty or a common vertex or edge, ensuring conformity of the mesh. For each element ${\cal K} \in \mathscr{T}_h$, let $\delta_{\cal K} := \mathrm{diam}({\cal K})$ denote its diameter. 
The global discretization parameter is defined as
$
h := \max_{{\cal K} \in \mathscr{T}_h} \delta_{\cal K}.
$ We assume that the mesh family $\{\mathscr{T}_h\}_{h>0}$ satisfies the standard regularity condition: there exist positive constants $\gamma_1$ and $\gamma_2$, independent of $h$, such that
$
\gamma_1 h \le \delta_{\cal K} \le \gamma_2 h,
\quad \forall {\cal K} \in \mathscr{T}_h.
$ This assumption guarantees uniform shape regularity of the elements. 
Consequently, the refinement process corresponds to the asymptotic regime $h \to 0$, leading to increasingly accurate spatial resolution.

\medskip
\noindent\textbf{Finite element approximation space:}
Associated with the triangulation $\mathscr{T}_h$, we introduce the finite-dimensional space
\[
W_h := \left\{ w_h \in C^0(\overline{\mathcal{D}}): 
w_h|_{K} \in \mathscr{P}_r(K), \ \forall K \in \mathcal{K}_h \right\},
\]
where $ \mathscr{P}_r(K)$ denotes the space of polynomials of degree at most $r$ on the element $K$.

Let $\{ \phi_i(\mathbf{x}) \}_{i=1}^{M}$ be the nodal basis of $W_h$, associated with the mesh nodes $\{N_i\}_{i=1}^{M}$. These basis functions satisfy $\phi_i(N_j) = \delta_{ij}, \quad 1 \le i,j \le M, $ and each $\phi_i$ has compact support restricted to the patch of elements sharing the node $N_i$.

\vspace{-0.34cm}
\section{Semi-discrete Galerkin Approximation}\label{sec:3}
\vspace{-0.24cm}

Let $\{\mathscr{T}_h\}_{h \to 0}$ be a quasi-uniform family of triangulations of $\mathcal{D}$ and 
$W_h \subset H_0^1(\mathcal{D} )$ be a finite element space spanned by the basis functions  $\{\varphi_i(x_1,x_2)\}_{i=1}^{M}.$
We approximate the exact solution $u(\mathbf{x},t)$ by the finite-dimensional expansion
\[
u_h(\mathbf{x},t) = \sum_{i=1}^{M} \alpha_i(t)\,\varphi_i(\mathbf{x}),
\]
where $\alpha_i(t)$ are unknown time-dependent coefficients and 
$
\boldsymbol{\alpha}(t) = (\alpha_1(t),\dots,\alpha_M(t))^{\top}
$ denotes the coefficient vector.

The semi-discretized finite element method is as follows: Given 
 $u_{0,h} \in W_h$, 
find $u_h \in L^2(0,T;W_h)$ such that
\begin{eqnarray}
(\partial_t u_h, \phi_h)
+ \mathcal{A}(u_h,\phi_h)
=
\mathcal{B}(u_h,\phi_h),
\qquad \forall \phi_h \in W_h,\; t>0, \label{semi1}
\end{eqnarray}
subject to the initial condition
\begin{eqnarray}
  u_h(0) = \mathcal{P}_h u_{\mathrm{in}}, \label{semi2}
\end{eqnarray}
where $(\cdot,\cdot)$ denotes the $L^2(D)$ inner product and 
$\mathcal{P}_h$ is an appropriate projection onto $W_h$.

\noindent
\textbf{Algebraic structure.}
Inserting the finite element expansion of $u_h$ and selecting 
$\phi_h=\varphi_j$, $j=1,\dots,M$, we obtain the system
\begin{eqnarray}
 \sum_{i=1}^{M} (\varphi_i,\varphi_j)\,\frac{d\alpha_i}{dt}
+
\sum_{i=1}^{M} \mathcal{A}(\varphi_i,\varphi_j)\,\alpha_i
=
\sum_{i=1}^{M} \mathcal{B}(\varphi_i,\varphi_j)\,\alpha_i,
\qquad j=1,\dots,M. \label{semi3}
   \end{eqnarray}

This yields the matrix formulation
\[
\mathbf{M}\,\boldsymbol{\alpha}'(t)
+
\mathbf{A}\,\boldsymbol{\alpha}(t)
=
\mathbf{B}\,\boldsymbol{\alpha}(t),
\]
or equivalently,
\[
\mathbf{M}\,\boldsymbol{\alpha}'(t)
=
(\mathbf{B}-\mathbf{A})\,\boldsymbol{\alpha}(t).
\]
\textbf{Definition of the matrices.}
\begin{itemize}
\item The \emph{mass matrix} $\mathbf{M} \in \mathbb{R}^{M\times M}$ is defined by
\[
\mathbf{M}_{ji}
=
(\varphi_i,\varphi_j)
=
\int_{D} \varphi_i(x_1,x_2)\,\varphi_j(x_1,x_2)\,dx_1dx_2.
\]
It is symmetric and positive definite due to the linear independence of the basis functions.

\item The \emph{selection (loss) matrix} $\mathbf{A}$ is given by
\[
\mathbf{A}_{ji}
=
\int_{D} {\Gamma}(x_1,x_2)\,
\varphi_i(x_1,x_2)\,
\varphi_j(x_1,x_2)\,dx_1dx_2,
\]
where ${\Gamma}(x_1,x_2)$ denotes the breakage rate.

\item The \emph{fragmentation (gain) matrix} $\mathbf{B}$ is defined as
\[
\mathbf{B}_{ji}
=
\int_{D}
\varphi_j(x_1,x_2)
\left(
\int_{x_1}^{x_{1\max}}
\int_{x_2}^{x_{2\max}}
{\beta}(x_1,x_2,y_1,y_2)\,
{\Gamma}(y_1,y_2)\,
\varphi_i(y_1,y_2)\,
dy_2dy_1
\right)
dx_2dx_1,
\]
where ${\beta}(x_1,x_2,y_1,y_2)$ denotes the fragmentation kernel.
\end{itemize}
\noindent
\textbf{Existence of the semi-discrete solution.}

Since the mass matrix $\mathbf{M}$ is symmetric and positive definite, it is invertible. Consequently, the semi-discrete system can be rewritten as
\[
\boldsymbol{\alpha}'(t)
=
\mathbf{M}^{-1}(\mathbf{B}-\mathbf{A})\,\boldsymbol{\alpha}(t).
\]
This represents a finite-dimensional system of linear ordinary differential equations with continuous coefficients. By classical ODE theory, for any prescribed initial vector 
$\boldsymbol{\alpha}(0)=\boldsymbol{\alpha}^0$, 
there exists a unique solution
$\boldsymbol{\alpha}(t) \in C^{1}([0,T];\mathbb{R}^{M}).
$ Hence, the semi-discrete Galerkin approximation $u_h(t)$ exists for all $t \in [0,T]$. Uniqueness follows directly from the system's stability properties. 

\vspace{-0.34cm}
\subsection*{Stability}
\vspace{-0.24cm}

We establish the stability of the semi-discrete finite element approximation
in the following lemma.

\begin{lemma}[Stability of the Semi-discrete Scheme]
Let $u_h(t) \in W_h$ be the semi-discrete Galerkin solution satisfying \eqref{semi1}
with initial data $u_h(0) \in W_h$.   Then the semi-discrete solution satisfies the stability estimate
\[
\|u_h(t)\| \le \|u_h(0)\| e^{b_0 t},
\qquad \forall t \in [0,T].
\]
\end{lemma}

\begin{proof}
Choosing $\phi_h = u_h$ in the semidiscrete formulation \eqref{semi1} gives
\[
(\partial_t u_h,u_h) + {\cal A}(u_h,u_h) = {\cal B}(u_h,u_h).
\]
We can rewrite as
\[
\frac{1}{2}\frac{d}{dt}\|u_h\|^2 + {\cal A}(u_h,u_h)
= {\cal B}(u_h,u_h).
\]
Applying the assumed bounds on ${\cal A}$ and ${\cal B}$, we directly deduce
\[
\frac{1}{2}\frac{d}{dt}\|u_h\|^2
\le b_0 \|u_h\|^2.
\]
Integrating this differential inequality over $[0,t]$ yields
\[
\|u_h(t)\|^2 \le \|u_h(0)\|^2 e^{2b_0 t},
\]
This completes the proof.
\end{proof}
\subsection*{Mass Conservation}
To establish mass conservation, we choose
\(\phi_j(x_1,x_2)=x_1+x_2\) in equation \eqref{semi3}, and we obtain
\[
\sum_{i=0}^{N} \frac{d\alpha_i}{dt}(\phi_i,x_1+x_2)
+ \sum_{i=0}^{N} \alpha_i\,{\cal A}(\phi_i,x_1+x_2)
=
\sum_{i=0}^{N} \alpha_i\,{\cal B}(\phi_i,x_1+x_2).
\]
Let \(D=(0,x_{1\max})\times(0,x_{2\max})\).
Using the definition of the inner product, we obtain
\begin{align*}
\sum_{i=0}^{N} \frac{d\alpha_i}{dt}
\int_D (x_1+x_2)\phi_i\,dx
=
\sum_{i=0}^{N} \alpha_i I_i^{\text{gain}}
-
\sum_{i=0}^{N} \alpha_i I_i^{\text{loss}},
\end{align*}
where
\[
I_i^{\text{gain}}
=
\int_D (x_1+x_2)
\int_{x_1}^{x_{1\max}}
\int_{x_2}^{x_{2\max}}
{\beta}(x_1,x_2,y_1,y_2)
{\Gamma}(y_1,y_2)
\phi_i(y_1,y_2)\,dy_2dy_1\,dx,
\]
\[
I_i^{\text{loss}}
=
\int_D (x_1+x_2)
{\Gamma}(x_1,x_2)
\phi_i(x_1,x_2)\,dx.
\]
Changing the order of integration and using property \eqref{eq:mass}, the gain and loss
terms cancel exactly, yielding
\[
\sum_{i=0}^{N} \frac{d\alpha_i}{dt}
\int_D (x_1+x_2)\phi_i\,dx = 0.
\]
Since \(u_h=\sum_{i=0}^{N}\alpha_i\phi_i\), it follows that
\[
\frac{d}{dt}
\int_D (x_1+x_2)u_h\,dx = 0.
\]
Hence, the total mass remains constant in time. The semi-discrete scheme
therefore preserves the fundamental mass invariant of the continuous model.

\subsection*{Number Preservation}
To analyze the particle number evolution, we take \(\phi_j(x_1,x_2)=1\) in equation \eqref{semi3}.
Substituting into the weak formulation gives
\begin{align*}
\sum_{i=0}^{N} \frac{d\alpha_i}{dt}
\int_D \phi_i\,dx
=
\sum_{i=0}^{N} \alpha_i J_i^{\text{gain}}
-
\sum_{i=0}^{N} \alpha_i J_i^{\text{loss}},
\end{align*}
where
\[
J_i^{\text{gain}}
=
\int_D
\int_{x_1}^{x_{1\max}}
\int_{x_2}^{x_{2\max}}
{\beta}(x_1,x_2,y_1,y_2)
{\Gamma}(y_1,y_2)
\phi_i(y_1,y_2)\,dy_2dy_1\,dx,
\]
\[
J_i^{\text{loss}}
=
\int_D
{\Gamma}(x_1,x_2)
\phi_i(x_1,x_2)\,dx.
\]
Reordering the integrals and using property \eqref{eq:multiplicity}, we obtain
\[
\sum_{i=0}^{N} \frac{d\alpha_i}{dt}
\int_D \phi_i\,dx
=
\sum_{i=0}^{N} \alpha_i
\int_D (\nu(y_1,y_2)-1)
{\Gamma}(y_1,y_2)
\phi_i(y_1,y_2)\,dy.
\]
Since \((\nu(y_1,y_2)-1)>0\), \({\Gamma}(y_1,y_2)>0\), and \(\phi_i\ge0\),
the right-hand side is nonnegative. Therefore,
\[
\frac{d}{dt}\int_D u_h\,dx \ge 0.
\]
Thus, while fragmentation increases the total number of particles,
the numerical scheme reproduces this growth exactly and introduces
no artificial loss or gain. The method is therefore both mass-conservative
and structurally consistent with the continuous fragmentation dynamics.

\subsection*{A Priori Error Estimate for the Semi-discrete Scheme}
We now derive an a priori error estimate for the semi-discrete Galerkin
approximation. The result quantifies the accuracy of the finite element
solution in terms of the mesh parameter $h$ and the regularity of the exact solution.

We recall the following standard result regarding the $L^2$ projection . Let ${\cal P}_h : L^2( \mathcal{D}) \to W_h$ denote the {$L^2$ projection} defined by
\[
({\cal P}_h u , \phi_h) = (u , \phi_h), \quad \forall \phi_h \in W_h.
\]
\begin{lemma}[Error Estimate for $L^2$ projection]\cite{thomee2007galerkin}\label{lem:ritz}
Let $\mathcal{P}_h$ denote the $L^2$ projection operator onto the finite element space $W_h$, which consists of piecewise polynomials of degree $r$ or less.   If  $ u \in H^{s+1}( \mathcal{D}), \quad 1 \le s \le r, $
then the following error estimate holds:
\begin{equation}\label{eq:ritz-error}
\| u - {\cal P}_h u \| \le C \, h^{s+1} \, \| u \|_{H^{s+1}( \mathcal{D})},
\end{equation}
where $h$ represents the mesh size and $C$ is a positive constant independent of $h$.
\end{lemma} \label{ritz}
\begin{thm}[Semi-discrete Error Estimate]
Let $u$ be the exact solution of the continuous problem and 
$u_h \in W_h$ be the semi-discrete Galerkin solution.  
Assume
$
u,\, u_t \in L^\infty(0,T; H^{r+1}(D)),\; r \ge 1.
$
Then, for every $t \in [0,T]$, the following estimate holds:
\[
\|u(t) - u_h(t)\|
\le
C h^{r+1}
\left(
\|u(t)\|_{H^{(r+1)}}
+
\int_0^t \big( \|u(s)\|_{H^{r+1}} + \|u_t(s)\|_{H^{r+1}} \big)\, ds
\right),
\]
where $C>0$ is independent of $h$.
\end{thm}

\begin{proof}
We decompose the error as
\[
u - u_h = (u - {\cal P}_h u) + ({\cal P}_h u - u_h)
= \eta + \theta,
\]
where
$\eta := u - {\cal P}_h u,
\;\;\theta := {\cal P}_h u - u_h.$

Subtracting the semi-discrete equation \eqref{semi1} for $u_h$ from the weak formulation \eqref{eq:weak_form}
satisfied by $u$, and using the definition of ${\cal P}_h$, we obtain
\[
(\theta_t,\phi_h) + {\cal A}(\theta,\phi_h)
=
{\cal B}(\theta,\phi_h)
+
{\cal B}(\eta,\phi_h)
-
(\eta_t,\phi_h)
-
{\cal A}(\eta,\phi_h),
\qquad \forall \phi_h \in W_h.
\]
taking $\phi_h = \theta$, we get
\[
(\theta_t,\theta) + {\cal A}(\theta,\theta)
=
{\cal B}(\theta,\theta)
+
{\cal B}(\eta,\theta)
-
(\eta_t,\theta)
-
{\cal A}(\eta,\theta).
\]
After simplification and 
together with the boundedness of ${\cal A}$ and ${\cal B}$, we obtain
\[
\frac{1}{2}\frac{d}{dt}\|\theta\|^2
\le
b_0 \|\theta\|^2
+
C\big(\|\eta\| + \|\eta_t\|\big)\|\theta\|.
\]
Hence,
\[
\frac{d}{dt}\|\theta\|
\le
b_0 \|\theta\|
+
C\big(\|\eta\| + \|\eta_t\|\big).
\]
Applying an integrating factor and integrating over $[0,t]$, we deduce
\[
\|\theta(t)\|
\le
e^{b_0 t}
\left(
\|\theta(0)\|
+
C \int_0^t
\big(
\|\eta(s)\| + \|\eta_t(s)\|
\big)\, ds
\right).
\]
Assuming the consistent initial condition $u_h(0)={\cal P}_h u(0)$,
we have $\theta(0)=0$, and therefore
\[
\|\theta(t)\|
\le
C e^{b_0 t}
\int_0^t
\big(
\|\eta(s)\| + \|\eta_t(s)\|
\big)\, ds.
\]
Using Lemma \ref{ritz}, we get
\[
\|\eta\| \le C h^{r+1} \|u\|_{r+1},
\qquad
\|\eta_t\| \le C h^{r+1} \|u_t\|_{r+1},
\]
we obtain
\[
\|\theta(t)\|
\le
C h^{r+1}
\int_0^t
\big(
\|u(s)\|_{r+1} + \|u_t(s)\|_{r+1}
\big)\, ds.
\]
Finally, by the triangle inequality,
\[
\|u(t)-u_h(t)\|
\le
\|\eta(t)\| + \|\theta(t)\|,
\]
which gives the desired estimate.
\end{proof}
\section{\large\textbf{Fully Discrete Error Analysis}}\label{sec:4}
\vspace{-0.4cm}

In this section, we derive the fully discrete formulation corresponding to the semi-discrete Galerkin approximation of the multidimensional fragmentation equation.
The spatial discretization is performed using the finite element space $W_h$,
while the temporal discretization is carried out using the BDF2 scheme. Let $T>0$ be the final time and let $N\in\mathbb{N}$ denote the number of time steps.
We define the uniform time step
\[
\tau=\frac{T}{N}, 
\qquad
t^n = n\tau, 
\qquad n=0,1,\dots,N.
\]
For a sufficiently smooth function 
$w:[0,T]\to L^2(D)$, we define
\[
w^n := w(\cdot,t^n),
\qquad
\delta_t w^n := \frac{w^n-w^{n-1}}{\tau},
\]
and the BDF2 operator
\[
D_t^{(2)} w^n
:=
\frac{3w^n-4w^{n-1}+w^{n-2}}{2\tau},
\qquad n\ge 2.
\]
The notation $w^n$ denotes the approximation at time level $t^n$,
while $D_t^{(2)} w^n$ provides a second-order approximation of
$\partial_t w(t^n)$.

\subsection*{BDF2 Fully Discrete Scheme}

The fully discrete finite element approximation seeks
$U_h^n \in W_h$ such that for all $\phi_h\in W_h$,
\begin{equation}
\label{FD-BDF2}
\begin{cases}
\big(D_t^{(2)} U_h^n,\phi_h\big)
+
{\cal A}(U_h^n,\phi_h)
=
{\cal B}(U_h^n,\phi_h),
& n\ge 2, \\[8pt]

\big(\delta_t U_h^1,\phi_h\big)
+
{\cal A}(u_h^1,\phi_h)
=
{\cal B}(U_h^1,\phi_h),
& n=1,\\[8pt]
U_h^0 = {\cal P}_hu_{\mathrm{in}},
\end{cases}
\end{equation}
The first time step is computed using the backward Euler method
to initialize the BDF2 scheme, ensuring second-order temporal accuracy
for $n\ge 2$. The scheme \eqref{FD-BDF2} constitutes the fully discrete finite element
approximation of the fragmentation model. In the subsequent analysis,
We establish its stability and derive error estimates under suitable
regularity assumptions on the exact solution.

\noindent
\textbf{Algebraic structure.}
Inserting the finite element expansion of $u_h^n$ and selecting 
$\phi_h=\varphi_j$, $j=1,\dots,M$, we obtain the system\\
For n=1,
\[
 \sum_{i=1}^{M} \frac{1}{\tau}(\varphi_i,\varphi_j)\,\alpha_i^n
+
\sum_{i=1}^{M} \mathcal{A}(\varphi_i,\varphi_j)\,\alpha_i^n
-
\sum_{i=1}^{M} \mathcal{B}(\varphi_i,\varphi_j)\,\alpha_i^n
= \sum_{i=1}^{M} \frac{1}{\tau}(\varphi_i,\varphi_j)\,\alpha_i^{n-1}
\qquad j=1,\dots,M.\quad 
\]

Therefore we get,
\[
\frac{\mathbf{M}}{\tau}\,\boldsymbol{\alpha}^n(t)
+
\mathbf{A}\,\boldsymbol{\alpha}^n(t)
-
\mathbf{B}\,\boldsymbol{\alpha}^n(t)
=\frac{1}{\tau}\,\boldsymbol{\alpha}^{n-1}(t)
\]
hence,
\begin{eqnarray}
(\frac{\mathbf{M}}{\tau}\
+
\mathbf{A}\
-
\mathbf{B})\,\boldsymbol{\alpha}^n(t)
=\frac{1}{\tau}\,\boldsymbol{\alpha}^{n-1}(t)\label{fully1}
   \end{eqnarray}
For $n\geq2$,
\[
\begin{aligned}
 \sum_{i=1}^{M} \frac{3}{2\tau}(\varphi_i,\varphi_j)\,\alpha_i^n
+
\sum_{i=1}^{M} \mathcal{A}(\varphi_i,\varphi_j)\,\alpha_i^n
-
\sum_{i=1}^{M} \mathcal{B}(\varphi_i,\varphi_j)\,\alpha_i^n
&= \sum_{i=1}^{M} \frac{2}{\tau}(\varphi_i,\varphi_j)\,\alpha_i^{n-1}  -\sum_{i=1}^{M} \frac{1}{2\tau}(\varphi_i,\varphi_j)\,\alpha_i^{n-2}, \\
&\qquad j = 1, \dots, M.
\end{aligned}
\]

This forms the matrix,
\[
\frac{\mathbf{3M}}{2\tau}\,\boldsymbol{\alpha}^n(t)
+
\mathbf{A}\,\boldsymbol{\alpha}^n(t)
-
\mathbf{B}\,\boldsymbol{\alpha}^n(t)
=\frac{2}{\tau}\,\boldsymbol{\alpha}^{n-1}(t)
-\frac{1}{2\tau}\,\boldsymbol{\alpha}^{n-2}(t)
\]

   \begin{eqnarray}
(\frac{3\mathbf{M}}{2\tau}\
+
\mathbf{A}\
-
\mathbf{B})\,\boldsymbol{\alpha}^n(t)
=\frac{2}{\tau}\,\boldsymbol{\alpha}^{n-1}(t)
-\frac{1}{2\tau}\,\boldsymbol{\alpha}^{n-2}(t)\label{fully2}
   \end{eqnarray}
\paragraph{Existence and Uniqueness of the Fully Discrete Scheme.}
The mass matrix $\mathbf{M}$ is symmetric and positive definite, therefore $\mathbf{M}$ is invertible. As we know that $\mathbf{A}$  and $\mathbf{B}$ are well defined. Hence, the operators are bounded. Also, for sufficiently small time step $\tau > 0$,  as the matrix $\mathbf{M}$ is positive definite, the given bilinear form is continuous and coercive. Hence, the matrix which needs to be invertible for BDF1 (n=1) is given by
\[
 \frac{\mathbf{M}}{\tau} + \mathbf{A} - \mathbf{B},
\]
and for the BDF2 scheme ($n \ge 2$), is
\[
\frac{3\mathbf{M}}{2\tau} + \mathbf{A} - \mathbf{B} .
\]

In both cases, as the positive-definite mass matrix dominates, this ensures that the system matrices are nonsingular. Hence, this gives the existence of a fully-discrete approximation for all $t \in [0, T]$. From the system's stability properties, the uniqueness follows.
\begin{lemma}[Stability of the fully discrete BDF2 scheme]
Let $\{U_h^n\}_{n\ge0}\subset W_h$ be the solution of the fully discrete
scheme \eqref{FD-BDF2}. Assume that $\Delta t < \frac{1}{4b_0}.$ Then the discrete solution satisfies
\[
\|U_h^{1}\|
\le (1-2b_0\Delta t)^{-1/2}\|U_h^{0}\|,
\quad
|||U_h^n|||
\le
\exp\!\left(\frac{2b_0 t_n}{1-4b_0\Delta t}\right)
|||U_h^{1}|||,
\quad t_n=n\Delta t,
\]
where $
|||U_h^n|||^2 := \|U_h^n\|^2 + \|2U_h^n-U_h^{\,n-1}\|^2.
$
Consequently, for any $0\le t_n\le T$,
\[
\|U_h^n\|\le C_T\|U_h^0\|,
\]
where $C_T>0$ is independent of $h$ and $\Delta t$. Hence, the fully discrete  BDF2 scheme is stable on any finite time interval.
\end{lemma}
\begin{proof}
For $n\ge2$, testing equation \eqref{FD-BDF2} with $\phi_h=U_h^n$ yields
\begin{eqnarray}
(D_t^{(2)}U_h^n,U_h^n\big) + {\cal A}(U_h^n,U_h^n) = {\cal B}(U_h^n,U_h^n). \label{eq1}
\end{eqnarray}
First, we observe that
\[
\left(
\frac{3U_h^n - 4U_h^{n-1} + U_h^{n-2}}{2\Delta t},
U_h^n
\right)
=
\frac{1}{2\Delta t}
\left(
3\|U_h^n\|^2
-4(U_h^{n-1},U_h^n)
+(U_h^{n-2},U_h^n)
\right).
\]
We can rewrite as 
\[\left(
\frac{3U_h^n - 4U_h^{n-1} + U_h^{n-2}}{2\Delta t},
U_h^n
\right)
=
\frac{1}{4\Delta t}
\left(
6\|U_h^n\|^2
-8(U_h^{n-1},U_h^n)
+2(U_h^{n-2},U_h^n)
\right).
\]
Using the identity
$
(a-2b+c)^2
=
a^2 + 4b^2 + c^2
-4ab -4bc +2ac,
$ with $a=U_h^n$, $b=U_h^{n-1}$, $c=U_h^{n-2}$, 
we obtain
\[
\left(
\frac{3U_h^n - 4U_h^{n-1} + U_h^{n-2}}{2\Delta t},
U_h^n
\right)
=
\frac{1}{4\Delta t}
\Big[
|||U_h^n|||^2
- |||U_h^{n-1}|||^2
+ \|U_h^n -2U_h^{n-1}+U_h^{n-2}\|^2
\Big].
\]
Substituting into \eqref{eq1} yields
\[
\frac{1}{4\Delta t}
\Big[
|||U_h^n|||^2
- |||U_h^{n-1}|||^2
+ \|U_h^n -2U_h^{n-1}+U_h^{n-2}\|^2
\Big]
+ {\cal A}(U_h^n,U_h^n)
=
{\cal B}(U_h^n,U_h^n).
\]

Since
$
\|U_h^n -2U_h^{n-1}+U_h^{n-2}\|^2 \ge 0.
$
We drop this nonnegative term and obtain
\[
\frac{|||U_h^n|||^2 - |||U_h^{n-1}|||^2}{4\Delta t}
+ {\cal A}(U_h^n,U_h^n)
\le
{\cal B}(U_h^n,U_h^n).
\]
Using coercivity and boundedness, and dropping the nonnegative term $a_0\|U_h^n\|^2$ gives
\[
|||U_h^n|||^2 - |||U_h^{n-1}|||^2
\le
4b_0\Delta t\,\|U_h^n\|^2.
\]
Since $\|U_h^n\|^2 \le |||U_h^n|||^2$, we can rewrite the above equation as
\[
(1-4b_0\Delta t)\,|||U_h^n|||^2
\le
|||U_h^{n-1}|||^2.
\]
For $\Delta t < \dfrac{1}{4b_0}$, the above iteration yields
\[
|||U_h^n|||^2
\le
(1-4b_0\Delta t)^{1-n}
|||U_h^1|||^2.
\]
Using simple calculus,
we obtain
\[
|||U_h^n|||
\le
\exp\!\left(
\frac{2b_0 t_n}{1-4b_0\Delta t}
\right)
|||U_h^1|||.
\]
Since $t_n\le T$, for sufficiently small $\Delta t$,
\[
\|U_h^n\|
\le
|||U_h^n|||
\le
\exp(2b_0T)\,|||U_h^1|||.
\]
For n=1,
\[
(\frac{U_h^{1}-U_h^{0}}{\Delta t},U_h^{1})
+ A(U_h^{1},U_h^{1})
= B(U_h^{1},U_h^{1}).
\]
From the stability estimate at the first time step, we obtain
\[
\frac{1}{2\Delta t}
\left(
\|U_h^{1}\|^2 - \|U_h^{0}\|^2
\right)
\le
b_0 \|U_h^{1}\|^2.
\]
Rearranging the above inequality gives
\[
\|U_h^{1}\|
\le
(1 - 2b_0 \Delta t)^{-1/2}
\|U_h^{0}\|.
\]
As for $n \ge 2$,
\[
\|U_h^{n}\|
\le
\exp(2b_0 T)
|||U_h^{1}|||.
\]
Since $t_n \le T$, and for sufficiently small $\Delta t$, we obtain
\[
\|U_h^{n}\|
\le
\exp(2b_0 T)
(\frac{1}{1-2b_0\Delta t}+(\frac{2}{\sqrt{1-2b_0\Delta t}}+1)^2)^{1/2}
\|U_h^{0}\|.
\]
In particular, using the bound we conclude
\[
\|U_h^n\|\le C_T\|U_h^0\|,
\]

\end{proof}

\vspace{-0.8cm}

\subsection*{Error Analysis of Fully-Discrete Formulation}
We now derive an a priori error estimate for the fully discrete scheme. We decompose the total error into two distinct parts: the spatial projection error and the temporal discretization error.
The error at time $t_n$ is then decomposed as:
\begin{equation} \label{3.7}
U_h^n - u^n = \theta^n + \eta^n,
\end{equation}
where
$\theta^n := U_h^n - {\cal P}_h u^n, \;\; 
\eta^n := {\cal P}_h u^n - u^n,$

\begin{thm}[\textbf{Fully Discrete Convergence for BDF2 Scheme}]\label{thm:bdf2_convergence}
Let $u$ be the solution of the linear parabolic problem \eqref{eq:model}, and let $U_h^n \in W_h$ denote the fully discrete finite element solution obtained using the BDF2 scheme \eqref{FD-BDF2}. Assume
\[
u \in H^1(0,T;H^{s+1}( \mathcal{D})) \cap H^3(0,T;L^2( \mathcal{D})), \quad s \in \{1,2,\ldots,r\}.
\]
Then, for sufficiently small $\Delta t$, the fully discrete error 
$ \mathbf{E}_h^n := U_h^n - u^n$
satisfies
\begin{equation}\label{eq:bdf2-conv}
\max_{1 \le n \le N} \| \mathbf{E}_h^n \|^2
\le C\left(
h^{2(s+1)}\Big(\|u_0\|_{H^{s+1}(\mathcal{D} )}^2
+\int_0^T \|u_t(s)\|_{H^{s+1}(\mathcal{D} )}^2\,ds\Big)
+\Delta t^4 \int_0^T \| u_{ttt}(s) \|^2\,ds
\right).
\end{equation}
where $C$ depends only on $T$ and the solution regularity, but is independent of $h$ and $\Delta t$.
\end{thm}
\begin{proof}
Let's recall the error splitting
$
\theta^n := U_h^n - {\cal P}_h u^n, \quad \eta^n := {\cal P}_h u^n - u^n.
$

Using Lemma \ref{lem:ritz}, we get
\begin{eqnarray} \label{tau3}
\|\eta^n\| \le Ch^{r+1}\|u(t_n)\|_{H^{r+1}(\mathcal{D} )}
\le Ch^{r+1} \left(\|u_0\|_{H^{r+1}(\mathcal{D} )}+\int_0^{t_n}\|u_t(s)\|_{H^{r+1}(\mathcal{D} )}ds\right).
\end{eqnarray}
Then, the fully discrete error equation reads in $\theta$ as
\begin{align}\label{eq:error_eq}
\left( \frac{3\theta^n - 4\theta^{n-1} + \theta^{n-2}}{2 \Delta t}, \phi_h \right) + {\cal A}(\theta^n, \phi_h)
= {\cal B}(\theta^n, \phi_h)+{\cal B}(\eta^n, \phi_h)\nonumber\\- {\cal A}(\eta^n, \phi_h) - (\tau^n, \phi_h),\quad \forall\;\phi_h \in W_h,\;n\geq 2.
\end{align}
The term $\tau^n$ denotes the temporal discretization error associated with 
the BDF2 approximation. It can be decomposed as
\[
\tau^n = D_t^{(2)}\eta(t_n) + \big(D_t^{(2)} u(t_n)-u_t(t_n)\big)
:= \tau_1^n + \tau_2^n,
\]
where
\[
\tau_1^n = D_t^{(2)}\eta(t_n), \qquad 
\tau_2^n = D_t^{(2)} u(t_n)-u_t(t_n).
\]
Using Taylor's theorem with integral remainder, the consistency error 
$\tau_2^n$ admits the representation
\[
\tau_2^n
=
-\frac{1}{2\Delta t}
\int_{t_{n-2}}^{t_n} (t_n-s)^2 u_{ttt}(s)\,ds .
\]
Consequently,
\[
|\tau_2^n|
\le
C\Delta t
\int_{t_{n-2}}^{t_n} |u_{ttt}(s)|\,ds .
\]
Multiplying by $\Delta t$ and summing over $n=2,\dots,N$, we obtain
\begin{eqnarray} \label{tau1}
\Delta t\sum_{n=2}^{N} \|\tau_2^n\|^2
\le
C\Delta t^4
\int_{0}^{t_N} \|u_{ttt}(s)\|^2\,ds .
\end{eqnarray}
For the term $\tau_1^n$, we have
\[
\tau_1^n=\frac{3\eta(t_n)-4\eta(t_{n-1})+\eta(t_{n-2})}{2\Delta t}.
\]
Using the identity
$
3\eta_n-4\eta_{n-1}+\eta_{n-2}
=
3(\eta_n-\eta_{n-1})-(\eta_{n-1}-\eta_{n-2}),
$ we obtain
\[
\Delta t\,\tau_1^n
=
\frac{3}{2}\int_{t_{n-1}}^{t_n}\eta_t(s)\,ds
-
\frac{1}{2}\int_{t_{n-2}}^{t_{n-1}}\eta_t(s)\,ds .
\]
Taking norms and using the triangle inequality yields
\[
\|\tau_1^n\|
\le
C \sup_{t\in(t_{n-2},t_n)} \|\eta_t(t)\|.
\]
Consequently,
\begin{eqnarray} \label{tau2}
\Delta t\sum_{n=2}^{N}\|\tau_1^n\|^2
\le
C h^{2(r+1)}\sum_{n=2}^{N}\int_{t_{n-2}}^{t_n}\|u_t(s)\|_{H^{r+1}}^2\,ds
\le
C h^{2(r+1)}\int_{0}^{t_N}\|u_t(s)\|_{H^{r+1}}^2\,ds .
\end{eqnarray}
Choosing $\phi_h=\theta^n$ in the error equation \eqref{eq:error_eq}, we obtain
\begin{align}
\left( \frac{3\theta^n-4\theta^{n-1}+\theta^{n-2}}{2\Delta t},\theta^n \right)
+{\cal A}(\theta^n,\theta^n)
&= {\cal B}(\theta^n,\theta^n)+{\cal B}(\eta^n,\theta^n) \nonumber\\
&\quad-{\cal A}(\eta^n,\theta^n)
+(\tau^n,\theta^n).
\end{align}
Next, we use the discrete BDF2 identity
\[
\left( \frac{3\theta^n - 4\theta^{n-1} + \theta^{n-2}}{2 \Delta t}, \theta^n \right)
= \frac{1}{4 \Delta t} 
\Big[ ||| \theta^n |||^2 - ||| \theta^{n-1} |||^2
+ \|\theta^n - 2\theta^{n-1} + \theta^{n-2}\|^2 \Big],
\]
where the discrete norm is defined as
$
|||\theta^n |||^2 := \|\theta^n\|^2 + \|2\theta^n - \theta^{n-1}\|^2. $

Substituting this identity into the above relation yields
\begin{align}
\frac{1}{4 \Delta t}
\Big[ ||| \theta^n |||^2 - ||| \theta^{n-1} |||^2
+ \|\theta^n - 2\theta^{n-1} + \theta^{n-2}\|^2 \Big]
+ {\cal A}(\theta^n,\theta^n)
= {\cal B}(\theta^n,\theta^n) \nonumber\\+ {\cal B}(\eta^n,\theta^n) 
 - {\cal A}(\eta^n,\theta^n)
+ (\tau^n,\theta^n).
\end{align}
Since the term $\|\theta^n - 2\theta^{n-1} + \theta^{n-2}\|^2$ is non–negative, it can be dropped to obtain the estimate
\begin{eqnarray}\label{eq:energy_ineq}
\left[ ||| \theta^n |||^2 - ||| \theta^{n-1} |||^2 \right] + {\cal A}(\theta^n, \theta^n)
\le 4\Delta t {\cal B}(\eta^n, \theta^n)-4 \Delta t{\cal A}(\eta^n, \theta^n)\nonumber\\ + 4 \Delta t(\tau^n, \theta^n)+4 \Delta t{\cal B}(\theta^n, \theta^n).
\end{eqnarray}
Summing \eqref{eq:energy_ineq} over $n=1,\dots,N$ and noting that the first
term forms a telescoping sum, we obtain
\begin{eqnarray} \label{lhs}
||| \theta^N |||^2 - ||| \theta^{0} |||^2 
+ \sum_{n=1}^{N} {\cal A}(\theta^n, \theta^n)
\le 4\Delta t \sum_{n=1}^{N} {\cal B}(\eta^n, \theta^n)
-4\Delta t \sum_{n=1}^{N} {\cal A}(\eta^n, \theta^n) \nonumber\\
+4\Delta t \sum_{n=1}^{N} (\tau^n, \theta^n)
+4\Delta t \sum_{n=1}^{N} {\cal B}(\theta^n, \theta^n).
\end{eqnarray}
The terms on the right-hand side of \eqref{eq:energy_ineq} can be estimated using the Cauchy--Schwarz inequality together with Young's inequality, with a suitable positive constant $\nu$, we obtain
\begin{eqnarray} \label{rhs}
4\Delta t \sum_{n=1}^{N}\Big[{\cal B}(\eta^n, \theta^n)-{\cal A}(\eta^n, \theta^n)
+ (\tau^n, \theta^n)
+{\cal B}(\theta^n, \theta^n)\Big]
\le 4\Delta t \sum_{n=1}^{N}\Big[C(\nu)\|\theta^n\|^2 \nonumber\\+C_{\nu}\|\tau^n\|^2+C^\nu\|\eta^n\|^2 \Big]
\le C \Delta t \sum_{n=1}^{N}\left(\|\theta^n\|^2+\|\tau^n_1\|^2+\|\tau^n_2\|^2+\|\eta^n\|^2\right).
\end{eqnarray}
Using the estimates \eqref{tau1} and \eqref{tau2} in the above inequality, we obtain
\begin{eqnarray}\label{rhs1}
4\Delta t \sum_{n=1}^{N}\Big[{\cal B}(\eta^n,\theta^n)-{\cal A}(\eta^n,\theta^n)
+(\tau^n,\theta^n)+{\cal B}(\theta^n,\theta^n)\Big]
\le C\Delta t \sum_{n=1}^{N}\|\theta^n\|^2 \nonumber\\
+ C h^{2(r+1)}\Big(\int_{0}^{t_N}\|u_t(s)\|_{H^{r+1}}^2\,ds 
+\Delta t\sum_{n=1}^{N}(\|u_0\|_{H^{r+1}(\mathcal D)} + \int_0^{t_n}\|u_t(s)\|_{H^{r+1}(\mathcal D)}\,ds)^2\Big)\nonumber\\
+ C\Delta t^4\int_{0}^{t_N}\|u_{ttt}(s)\|^2\,ds .
\end{eqnarray}
Inserting the estimate \eqref{rhs1} into \eqref{lhs}, we obtain
\begin{align}
||| \theta^N |||^2 - ||| \theta^{0} |||^2 
+ \sum_{n=1}^{N} {\cal A}(\theta^n, \theta^n)
&\le C\Delta t \sum_{n=1}^{N}\|\theta^n\|^2
+ C h^{2(r+1)}\Big(\int_{0}^{t_N}\|u_t(s)\|_{H^{r+1}}^2\,ds \nonumber\\
&\qquad +\Delta t\sum_{n=1}^{N}(\|u_0\|_{H^{r+1}(\mathcal D)} + \int_0^{t_n}\|u_t(s)\|_{H^{r+1}(\mathcal D)}\,ds)^2\Big)\nonumber\\
&\qquad + C\Delta t^4\int_{0}^{t_N}\|u_{ttt}(s)\|^2\,ds .
\end{align}
Since $\theta^0=0$ and $\|\theta^n\|^2 \le ||| \theta^n |||^2$, the above inequality simplifies to
\begin{eqnarray}
||| \theta^N |||^2 
+ \sum_{n=1}^{N} {\cal A}(\theta^n, \theta^n)
\le C \Delta t \sum_{n=1}^{N}|||\theta^n|||^2+
 C h^{2(r+1)}\Big(\int_{0}^{t_N}\|u_t(s)\|_{H^{r+1}}^2\,ds \nonumber\\
 +\Delta t\sum_{n=1}^{N}(\|u_0\|_{H^{r+1}(\mathcal D)} + \int_0^{t_n}\|u_t(s)\|_{H^{r+1}(\mathcal D)}\,ds)^2\Big)
+ C\Delta t^4\int_{0}^{t_N}\|u_{ttt}(s)\|^2\,ds .
\end{eqnarray}
Finally, applying the discrete Grönwall lemma, we conclude
\begin{eqnarray} \label{t2}
||| \theta^n |||^2 \le C h^{{2(r+1)}}(\int_{0}^{t_N}\|u_t(s)\|_{H^{r+1}}^2\,ds+\|u_0\|_{H^{r+1}}^2)+C\Delta t^4
\int_{0}^{t_N} \|u_{ttt}(s)\|^2\,ds.
\end{eqnarray}
By combining the estimates \eqref{tau3} and \eqref{t2}, we arrive at the desired fully discrete estimate
\begin{eqnarray}
\|U_h^n - u^n\|^2 
&\le& C\left(\|\theta^n\|^2+\|\eta^n\|^2\right) \nonumber\\
&\le& C h^{2(r+1)}\left(\|u_0\|_{H^{r+1}}^2
+\int_{0}^{t_N}\|u_t(s)\|_{H^{r+1}}^2\,ds\,\right) \nonumber\\
&&+C\Delta t^4\int_{0}^{t_N}\|u_{ttt}(s)\|^2\,ds .
\end{eqnarray}
which completes the proof.
\end{proof}

\section{Numerical Results}\label{sec:5}

This section checks the numerical performance of the proposed finite element scheme in a computational domain which is taken to be bounded $\mathcal{D} \subset \mathbb{R}^d$, $d \geq 1$. The main objective is to check the convergence properties and accuracy of the spatial discretization. The domain $\mathcal{D}$ is discretized by both meshes, i.e., uniform and non-uniform graded meshes, to check the influence caused by mesh refinement on approximation quality. Moreover, we also compute the moments associated with the considered integro–partial differential equation. 
Let $u$ denote the exact analytical solution and $u_h$ the corresponding finite element approximation. The discretization error is defined as $\mathcal{E}:= u - u_h.$ The error is measured in the standard norms $L^2(\mathcal{D} )$ and $H^1(\mathcal{D} )$. The $L^2$-error is defined as
\begin{equation}
\|\mathcal{E}\|_{L^2(\mathcal{D} )}
=
\left(
\int_{\mathcal{D} } |\mathcal{E}|^2 \, d\mathbf{x}
\right)^{1/2},
\quad
\text{where } \mathbf{x}=(x_1,x_2,\dots,x_d)\in\mathbb{R}^d.
\end{equation}
The $H^1$-error is given by
$
\|\mathcal{E}\|_{H^1(\mathcal{D} )}
=
\left(
\|\mathcal{E}\|_{L^2(\mathcal{D} )}^2
+
\sum_{i=1}^{d}
\left\|
\frac{\partial \mathcal{E}}{\partial x_i}
\right\|_{L^2(\mathcal{D} )}^2
\right)^{1/2}.
$ 

To obtain a normalized measure of the approximation quality, we also compute the relative $L^2$-error defined by
\begin{equation}
\mathrm{RelError}_{L^2}
=\frac{\|u-u_h\|_{L^2(\mathcal{D} )}}
     {\|u\|_{L^2(\mathcal{D} )}}.
\end{equation}
To quantify the asymptotic behavior of the scheme, the Experimental Order of Convergence (EOC) is computed at the final simulation time $T$. When the exact solution $u$ is available, the EOC is evaluated by comparing the errors obtained on two successive mesh refinements characterized by mesh sizes $h_N$ and $h_{2N}$:
\begin{equation}
\mathrm{EOC}
=
\frac{\ln \left( \mathcal{E}_N / \mathcal{E}_{2N} \right)}
     {\ln \left( h_N / h_{2N} \right)}.
\end{equation}
\subsection{Moments and their relative error}
Let $\mathcal{M}(t)$ denote an exact analytical moment of the integro–partial differential equation and $\mathcal{M}_h(t)$ its numerical approximation computed from $u_h$. The relative error of the moment is defined as
\begin{equation}
\mathrm{RelError}_{\mathcal{M}}(t)
=
\frac{|\mathcal{M}(t)-\mathcal{M}_h(t)|}
     {|\mathcal{M}(t)|}.
\end{equation}
These quantities allow us to assess both local accuracy (via norm errors) and global structural accuracy (via moment preservation). For the validation of the numerical scheme, several representative test cases have been considered for the computation of the analytical and numerical moments of the integro–partial differential equation. The comparison between exact and computed moments allows a detailed assessment of the accuracy and consistency of the proposed finite element approximation. \\
The selected test cases, along with their exact moment formulas, are summarized in Table~\ref{tab:test_cases}. All the problems are considered with mono-disperse initial data \(u(x_1, x_2, 0) = \delta(x_1 - 1)\delta(x_2 - 1)\). The computational domain is chosen as \(\mathcal{D} := [10^{-9}, 2]\). For the 2D cases , the domain is discretized using a $20 \times 20$ nonuniform grid. To investigate the effect of discretization errors, additional simulations are performed on finer grids consisting of $40 \times 40$ and $60 \times 60$ nonuniform cells.

\begin{table}[htbp]
\centering
\caption{Summary of the selected test problems in two and three dimensions.}
\label{tab:test_cases}
\renewcommand{\arraystretch}{1.4}
\setlength{\tabcolsep}{6pt}

\begin{tabular}{@{}c c c p{7.2cm}@{}}
\toprule
\textbf{Test} 
& $\boldsymbol{\Gamma(x_1,x_2,x_3)}$ 
& $\boldsymbol{\beta(x \mid y)}$ 
& \textbf{Exact moments} \\
\midrule

\multicolumn{4}{c}{\textbf{Two–Dimensional Test Cases}} \\
\midrule

1 
& $1$ 
& $\dfrac{2}{y_1 y_2}$ 
& $\displaystyle 
\mathcal{M}_{k,l}(t)=
\exp\!\left[\left(\frac{2}{(k+1)(l+1)}-1\right)t\right]$ \\[8pt]

2 
& $x_1+x_2$ 
& $\dfrac{2}{y_1 y_2}$ 
& $\displaystyle 
\mathcal{M}_{1,0}(t)=\mathcal{M}_{0,1}(t)=1, 
\qquad
\mathcal{M}_{0,0}(t)=1+2t$ \\[8pt]

3 
& $1$ 
& $\displaystyle 
2\delta\!\left(x_1-\frac{y_1}{2}\right)
\delta\!\left(x_2-\frac{y_2}{2}\right)$ 
& $\displaystyle 
\mathcal{M}_{k,l}(t)=
\exp\!\left[(2^{\,1-k-l}-1)t\right]$ \\[8pt]

4 
& $x_1+x_2$ 
& $\displaystyle 
2\delta\!\left(x_1-\frac{y_1}{2}\right)
\delta\!\left(x_2-\frac{y_2}{2}\right)$ 
& $\displaystyle 
\mathcal{M}_{1,0}(t)=\mathcal{M}_{0,1}(t)=1,
\qquad
\mathcal{M}_{0,0}(t)=1+2t$ \\

\midrule
\multicolumn{4}{c}{\textbf{Three–Dimensional Test Case}} \\
\midrule

5 
& $x_1 x_2 x_3$ 
& $\dfrac{8}{y_1 y_2 y_3}$ 
& $\displaystyle 
\mathcal{M}_{1,1,1}(t)=1,
\qquad
\mathcal{M}_{0,0,0}(t)=1+7t$ \\

\bottomrule
\end{tabular}
\end{table}


\textbf{Test Case 1: }
In this example, fragmentation is binary, showing a uniform distribution of daughter particles with a constant selection function. This implies that all particle sizes break at the same rate. Each fragment forms independently, and the resulting daughter particles are uniformly distributed. The results for Case 1 are illustrated in Figure~\ref{fig:case1}. Mass is conserved throughout the simulation. The accuracy of the zeroth moment improves as the grid is refined. The relative errors of the moments computed using the BDF2 scheme are listed in Table~\ref{tab:m00} - Table~\ref{tab:m10m01} for different grid resolutions. It can be observed that the relative errors in all moments decrease as we refine the grids.

\begin{table}[htbp]
\centering
\caption{Comparison of exact and numerical values of $m_{00}(t)=e^t$.}
\label{tab:m00}
\renewcommand{\arraystretch}{1.35}
\setlength{\tabcolsep}{8pt}

\begin{tabular}{c c cc cc cc}
\toprule
$t$ & Exact 
& \multicolumn{2}{c}{$20\times20$}
& \multicolumn{2}{c}{$40\times40$}
& \multicolumn{2}{c}{$60\times60$} \\

\cmidrule(lr){3-4} \cmidrule(lr){5-6} \cmidrule(lr){7-8}

& 
& Num & Rel. Err.
& Num & Rel. Err.
& Num & Rel. Err. \\

\midrule

0.1 & 1.10517
& 1.11161 & 5.82e-03
& 1.10655 & 1.25e-03
& 1.10565 & 4.32e-04 \\

0.4 & 1.49182
& 1.52020 & 1.90e-02
& 1.49589 & 2.73e-03
& 1.49184 & 7.62e-06 \\

0.7 & 2.01375
& 2.24506 & 1.15e-01
& 2.04005 & 1.31e-02
& 2.01690 & 1.56e-03 \\

1.0 & 2.71828
& 11.64820 & 3.29e+00
& 3.36613 & 2.38e-01
& 2.85339 & 4.97e-02 \\

\bottomrule
\end{tabular}
\end{table}
\begin{table}[htbp]
\centering
\caption{Comparison of exact and numerical values of $m_{10}(t)+m_{01}(t)=2$.}
\label{tab:m10m01}
\renewcommand{\arraystretch}{1.35}
\setlength{\tabcolsep}{8pt}

\begin{tabular}{c c cc cc cc}
\toprule
$t$ & Exact 
& \multicolumn{2}{c}{$20\times20$}
& \multicolumn{2}{c}{$40\times40$}
& \multicolumn{2}{c}{$60\times60$} \\

\cmidrule(lr){3-4} \cmidrule(lr){5-6} \cmidrule(lr){7-8}

& 
& Num & Rel. Err.
& Num & Rel. Err.
& Num & Rel. Err. \\

\midrule

0.1 & 2
& 2.03174 & 1.59e-02
& 2.00774 & 3.87e-03
& 2.00334 & 1.67e-03 \\

0.4 & 2
& 2.04671 & 2.34e-02
& 2.01073 & 5.36e-03
& 2.00435 & 2.18e-03 \\

0.7 & 2
& 2.06744 & 3.37e-02
& 2.01499 & 7.49e-03
& 2.00590 & 2.95e-03 \\

1.0 & 2
& 2.09929 & 4.96e-02
& 2.02122 & 1.06e-02
& 2.00826 & 4.13e-03 \\

\bottomrule
\end{tabular}
\end{table}
\begin{figure}[htbp]
\centering
\begin{subfigure}{0.32\textwidth}
    \centering
    \includegraphics[width=\linewidth]{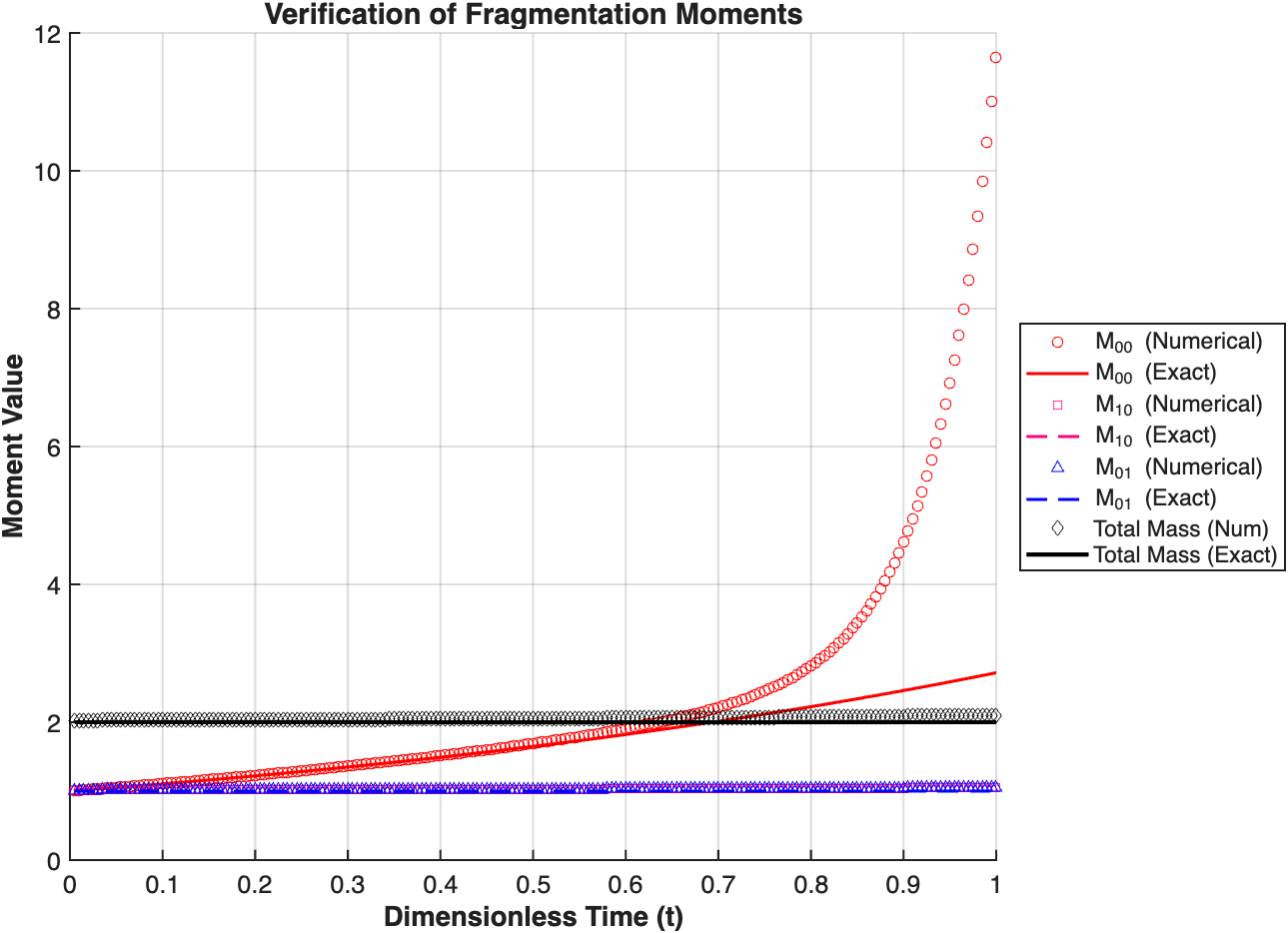}
    \caption{$20\times20$ grid}
\end{subfigure}
\hfill
\begin{subfigure}{0.32\textwidth}
    \centering
    \includegraphics[width=\linewidth]{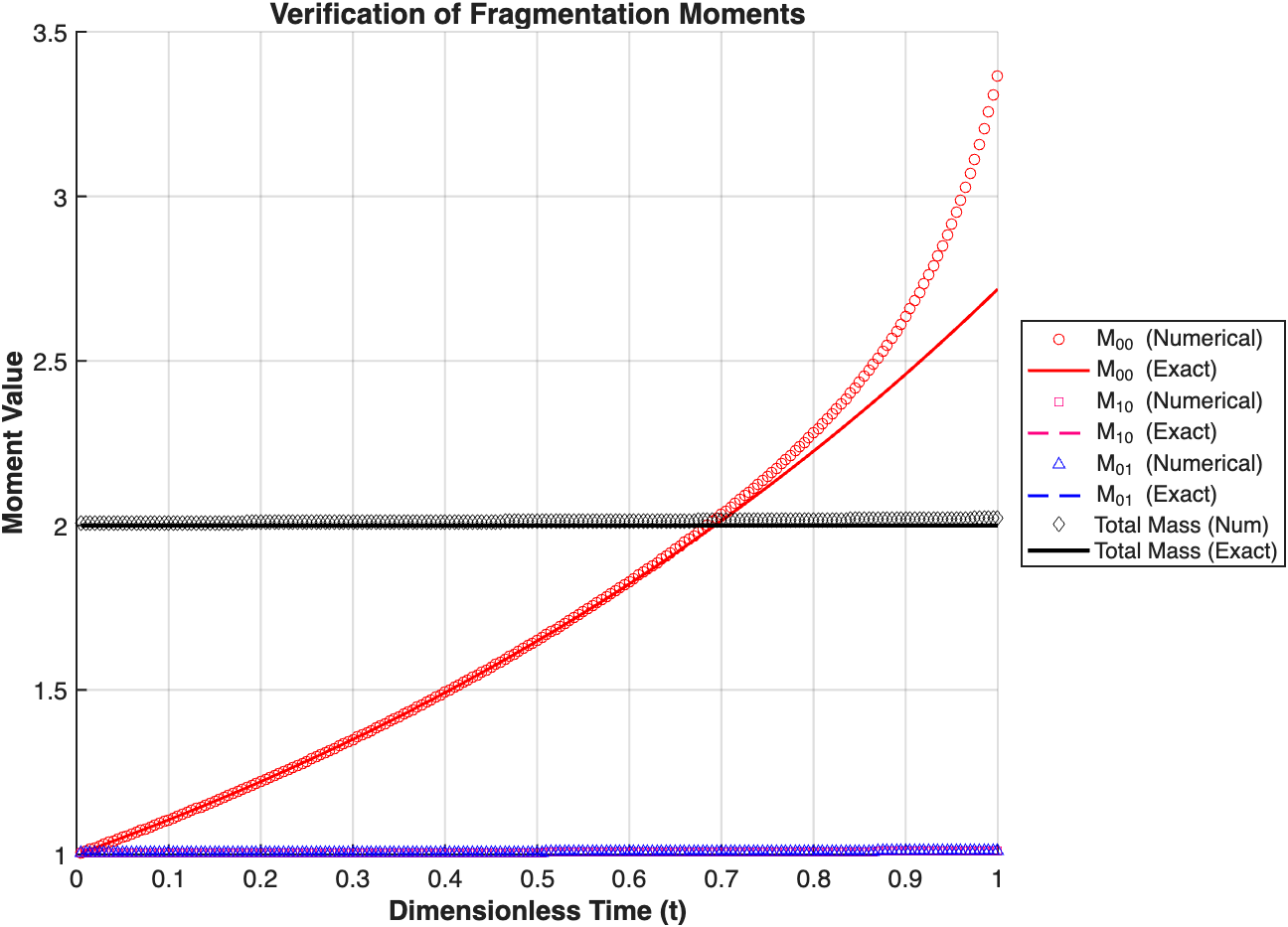}
    \caption{$40\times40$ grid}
\end{subfigure}
\hfill
\begin{subfigure}{0.32\textwidth}
    \centering
    \includegraphics[width=\linewidth]{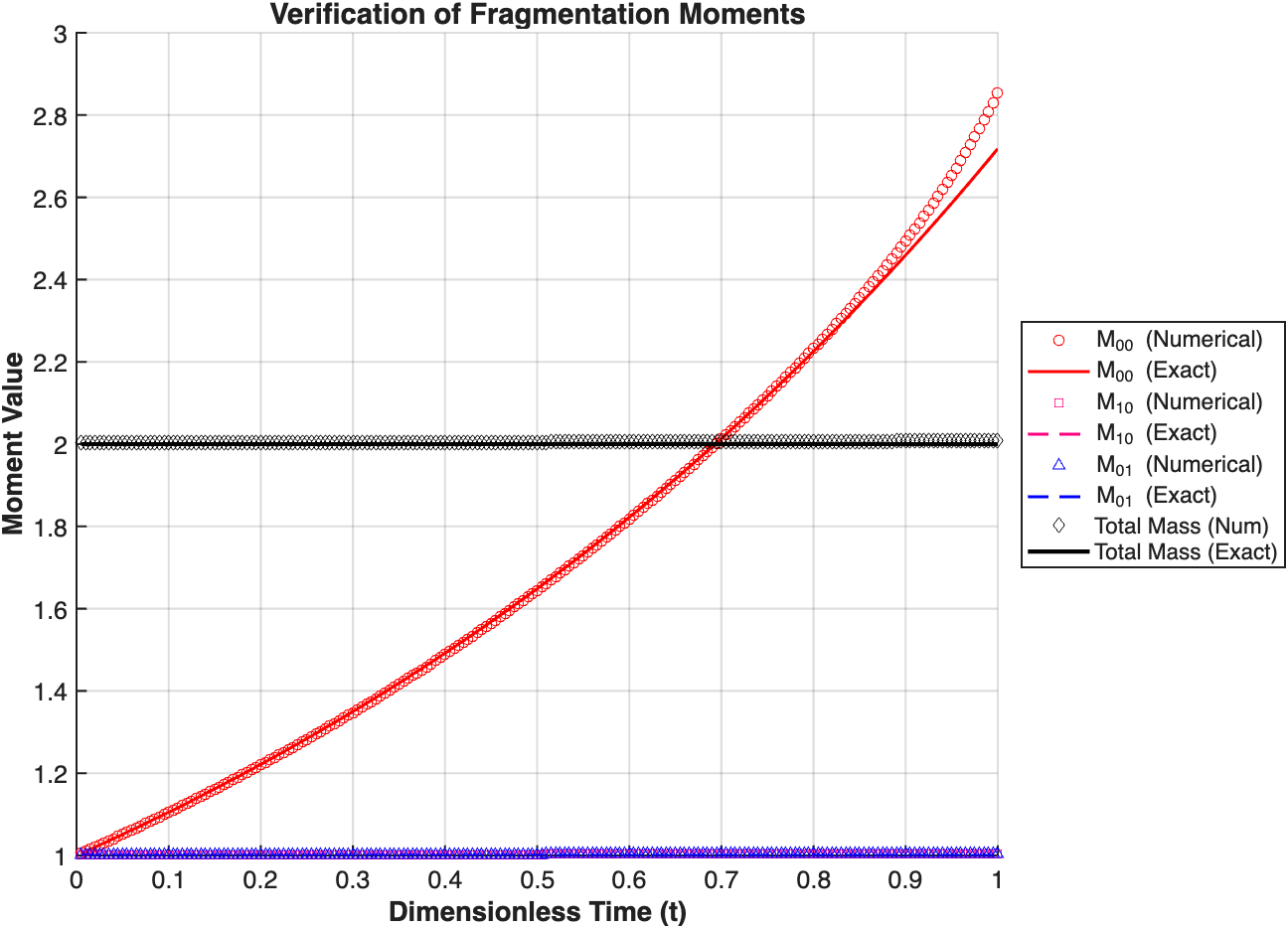}
    \caption{$60\times60$ grid}
\end{subfigure}

\caption{Comparison of Test Case 1 for different grid resolutions.}
\label{fig:case1}
\end{figure}

\textbf{Test Case 2: }
In this example, the fragmentation process is binary with a uniform distribution of daughter particles. The breakage mechanism is governed by a size-dependent selection function given by $\Gamma(x_1,x_2) = x_1 + x_2$, implying that the breakage rate is proportional to the size of the particle. The simulations are performed over the time interval $t \in [0,3]$. The numerical results are summarized in Table~\ref{tab:m2_m00} -Table~\ref{tab:m2_m10m01},  and the corresponding solution profiles for Case 2 are displayed in Figure~\ref{fig:case2}. Mass conservation is preserved throughout the simulation. Moreover, the mixed-order moment is accurately captured by the proposed numerical scheme. The relative errors in all moments decrease considerably as the mesh is refined.

\begin{table}[htbp]
\centering
\caption{Comparison of exact and numerical values of $m_{00}(t)=1+2t$ for Test Case 2.}
\label{tab:m2_m00} 
\renewcommand{\arraystretch}{1.35}
\setlength{\tabcolsep}{8pt}

\begin{tabular}{c c cc cc cc}
\toprule
$t$ & Exact
& \multicolumn{2}{c}{$20\times20$}
& \multicolumn{2}{c}{$40\times40$}
& \multicolumn{2}{c}{$60\times60$} \\

\cmidrule(lr){3-4} \cmidrule(lr){5-6} \cmidrule(lr){7-8}

& 
& Num & Rel. Err.
& Num & Rel. Err.
& Num & Rel. Err. \\
\midrule

1.0 & 3
& 3.16621 & 5.54e-02
& 3.04374 & 1.46e-02
& 3.02325 & 7.75e-03 \\

1.5 & 4
& 4.34525 & 8.63e-02
& 4.09869 & 2.47e-02
& 4.05770 & 1.44e-02 \\

2.0 & 5
& 5.60887 & 1.22e-01
& 5.18265 & 3.65e-02
& 5.11240 & 2.25e-02 \\

2.5 & 6
& 6.97454 & 1.62e-01
& 6.29981 & 4.99e-02
& 6.18984 & 3.16e-02 \\

3.0 & 7
& 8.46356 & 2.09e-01
& 7.45446 & 6.49e-02
& 7.29225 & 4.18e-02 \\

\bottomrule
\end{tabular}
\end{table}

\begin{table}[htbp]
\centering
\caption{Comparison of exact and numerical values of $m_{10}(t)+m_{01}(t)=2$ for Test Case 2.}
\label{tab:m2_m10m01}
\renewcommand{\arraystretch}{1.35}
\setlength{\tabcolsep}{8pt}

\begin{tabular}{c c cc cc cc}
\toprule
$t$ & Exact
& \multicolumn{2}{c}{$20\times20$}
& \multicolumn{2}{c}{$40\times40$}
& \multicolumn{2}{c}{$60\times60$} \\

\cmidrule(lr){3-4} \cmidrule(lr){5-6} \cmidrule(lr){7-8}

& 
& Num & Rel. Err.
& Num & Rel. Err.
& Num & Rel. Err. \\
\midrule

1.0 & 2
& 2.12945 & 6.47e-02
& 2.04315 & 2.16e-02
& 2.02863 & 1.43e-02 \\

1.5 & 2
& 2.19216 & 9.61e-02
& 2.06912 & 3.46e-02
& 2.04870 & 2.43e-02 \\

2.0 & 2
& 2.26201 & 1.31e-01
& 2.09766 & 4.88e-02
& 2.07072 & 3.54e-02 \\

2.5 & 2
& 2.34005 & 1.70e-01
& 2.12836 & 6.42e-02
& 2.09413 & 4.71e-02 \\

3.0 & 2
& 2.42789 & 2.14e-01
& 2.16117 & 8.06e-02
& 2.11872 & 5.94e-02 \\

\bottomrule
\end{tabular}
\end{table}

\begin{figure}[htbp]
\centering

\begin{subfigure}{0.32\textwidth}
    \centering
    \includegraphics[width=\linewidth]{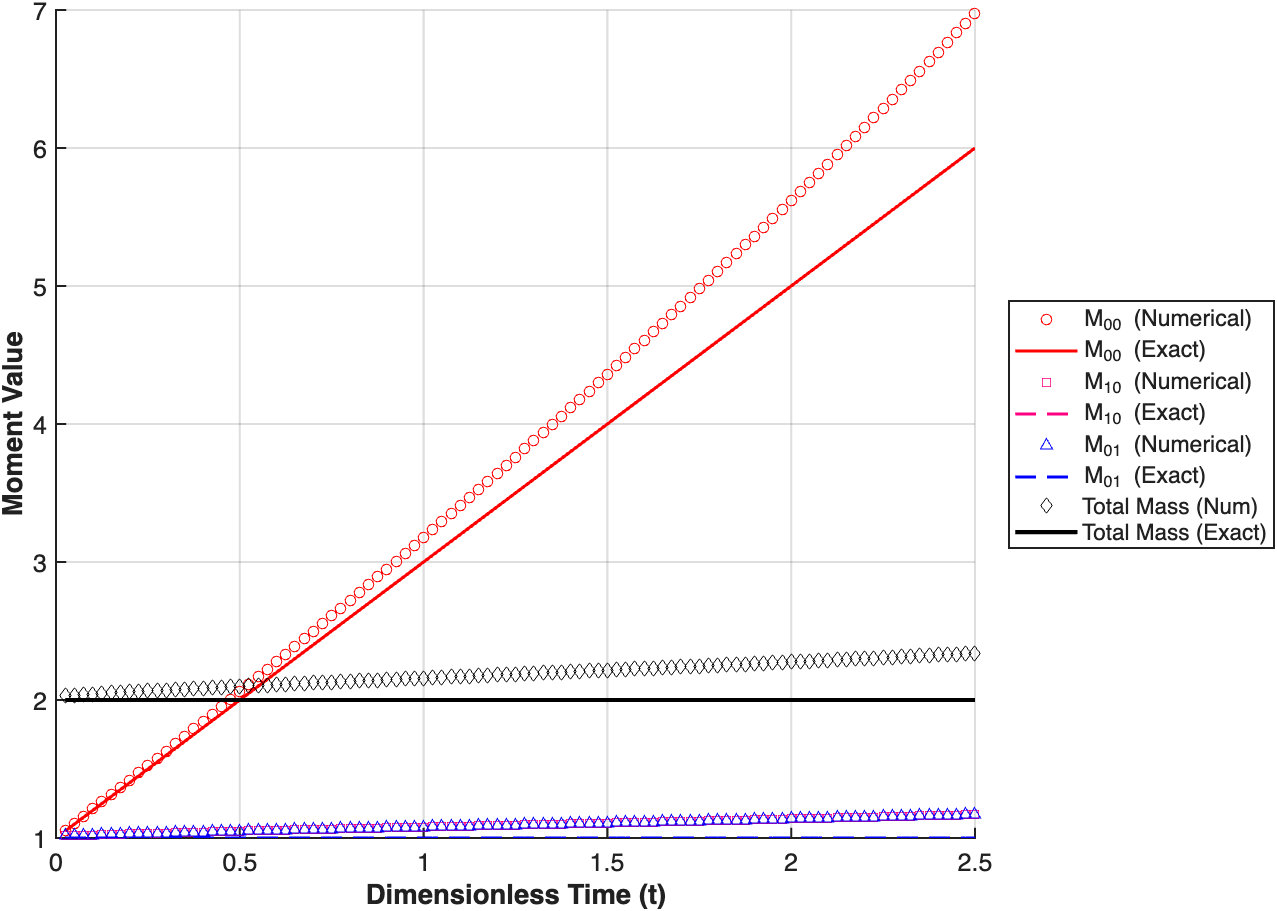}
    \caption{$20\times20$ grid}
\end{subfigure}
\hfill
\begin{subfigure}{0.32\textwidth}
    \centering
    \includegraphics[width=\linewidth]{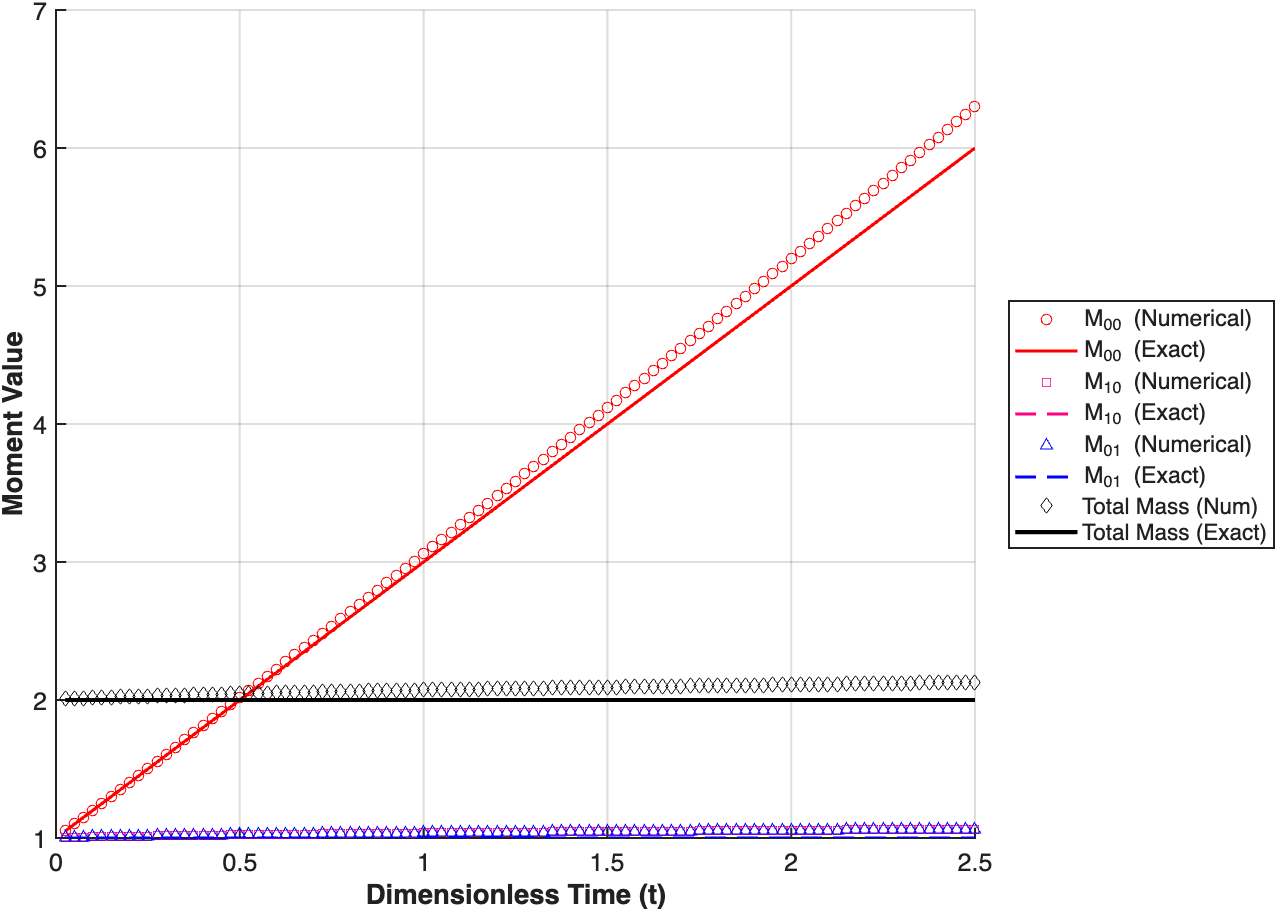}
    \caption{$40\times40$ grid}
\end{subfigure}
\hfill
\begin{subfigure}{0.32\textwidth}
    \centering
    \includegraphics[width=\linewidth]{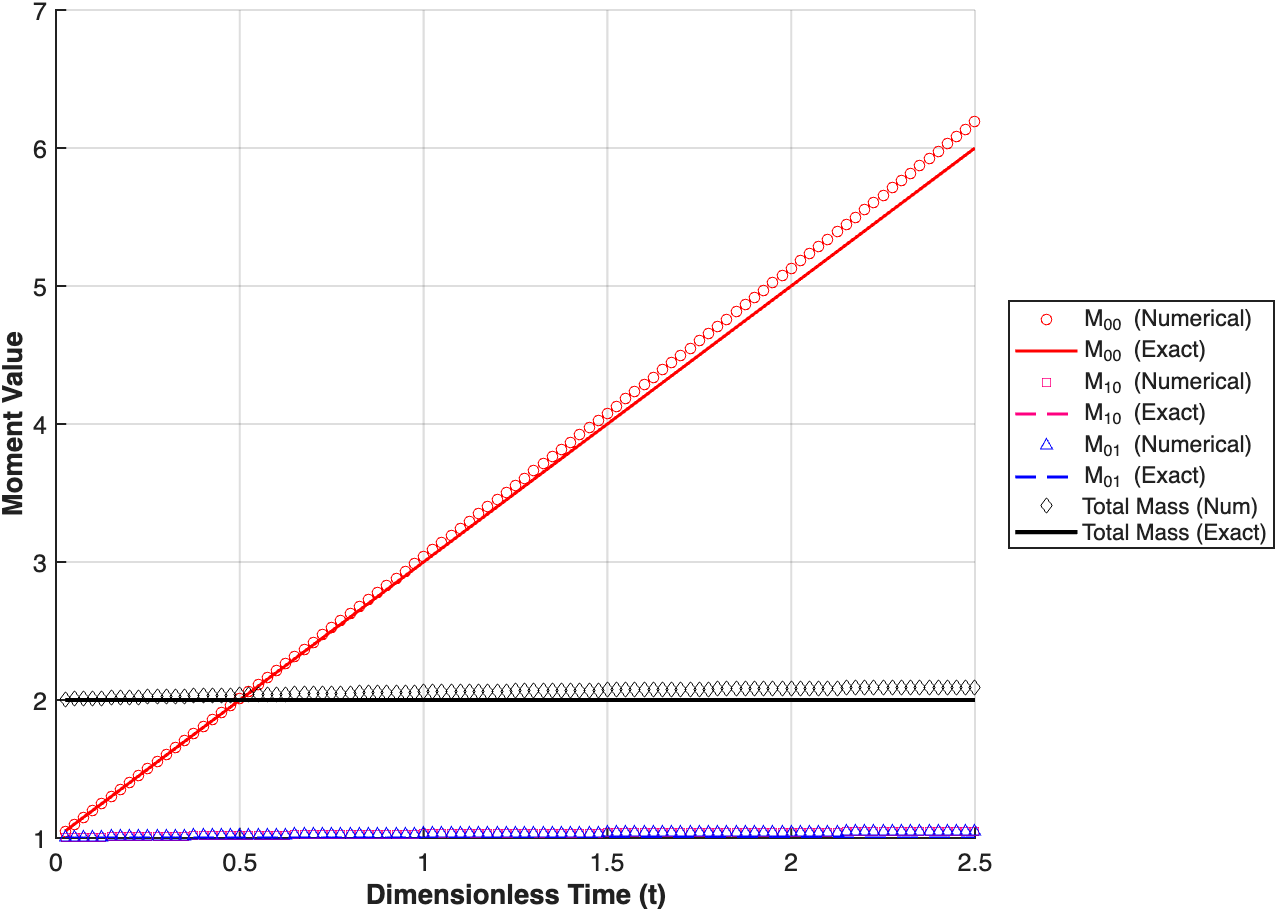}
    \caption{$60\times60$ grid}
\end{subfigure}

\caption{Comparison of Test Case 2 for different grid resolutions.}
\label{fig:case2}
\end{figure}
\textbf{Test Case 3: }In this example, the fragmentation process is binary with a constant selection rate equal to $\Gamma=1$. The breakage kernel is defined as  $2\delta\!\left(x_1-\frac{y_1}{2}\right)
\delta\!\left(x_2-\frac{y_2}{2}\right),
$ which represents symmetric binary breakage, where each parent particle splits into two equal fragments. The fragmentation of each particle occurs independently. The simulations are performed over the time interval $t \in [0,3]$. The numerical results are reported in Table~\ref{tab:m3_m00} - Table~\ref{tab:m3_m10m01}, while the corresponding solution profiles for Case 3 are presented in Figure~\ref{fig:case3}. Mass conservation is preserved throughout the simulation.
\begin{table}[H]
\centering
\caption{Comparison of exact and numerical values of $m_{00}(t)=e^t$ for Test Case 3.}
\label{tab:m3_m00}
\renewcommand{\arraystretch}{1.35}
\setlength{\tabcolsep}{8pt}

\begin{tabular}{c c cc cc cc}
\toprule
$t$ & Exact
& \multicolumn{2}{c}{$20\times20$}
& \multicolumn{2}{c}{$40\times40$}
& \multicolumn{2}{c}{$60\times60$} \\
\cmidrule(lr){3-4} \cmidrule(lr){5-6} \cmidrule(lr){7-8}
& 
& Num & Rel. Err.
& Num & Rel. Err.
& Num & Rel. Err. \\
\midrule

1.0 & 2.71828
& 2.31496 & 1.48e-01
& 2.58996 & 4.72e-02
& 2.63283 & 3.14e-02 \\

1.5 & 4.48169
& 3.01977 & 3.26e-01
& 3.85894 & 1.39e-01
& 4.05629 & 9.49e-02 \\

2.0 & 7.38906
& 3.58432 & 5.15e-01
& 5.35539 & 2.75e-01
& 5.91338 & 2.00e-01 \\

2.5 & 12.1825
& 3.95469 & 6.75e-01
& 6.94802 & 4.30e-01
& 8.13667 & 3.32e-01 \\

3.0 & 20.0855
& 4.13430 & 7.94e-01
& 8.50406 & 5.77e-01
& 10.6114 & 4.72e-01 \\

\bottomrule
\end{tabular}
\end{table}
\begin{table}[H]
\centering
\caption{Comparison of exact and numerical values of $m_{10}(t)+m_{01}(t)=2$ for Test Case 3.}
\label{tab:m3_m10m01}
\renewcommand{\arraystretch}{1.35}
\setlength{\tabcolsep}{8pt}

\begin{tabular}{c c cc cc cc}
\toprule
$t$ & Exact
& \multicolumn{2}{c}{$20\times20$}
& \multicolumn{2}{c}{$40\times40$}
& \multicolumn{2}{c}{$60\times60$} \\
\cmidrule(lr){3-4} \cmidrule(lr){5-6} \cmidrule(lr){7-8}
& 
& Num & Rel. Err.
& Num & Rel. Err.
& Num & Rel. Err. \\
\midrule

1.0 & 2
& 1.93237 & 3.38e-02
& 1.99562 & 2.19e-03
& 1.99982 & 9.18e-05 \\

1.5 & 2
& 1.84533 & 7.73e-02
& 1.98477 & 7.61e-03
& 1.99874 & 6.28e-04 \\

2.0 & 2
& 1.72423 & 1.38e-01
& 1.96320 & 1.84e-02
& 1.99555 & 2.22e-03 \\

2.5 & 2
& 1.57697 & 2.12e-01
& 1.92811 & 3.59e-02
& 1.98890 & 5.55e-03 \\

3.0 & 2
& 1.41447 & 2.93e-01
& 1.87830 & 6.09e-02
& 1.97771 & 1.11e-02 \\

\bottomrule
\end{tabular}
\end{table}

%
%

\textbf{Test Case 4: }In this example, the fragmentation mechanism is binary with a size-dependent selection function given by $
S(x_1,x_2) = x_1 + x_2,$
indicating that a particles with larger total size undergo breakage at a higher rate. The breakage kernel is defined as $
2\delta\!\left(x_1-\frac{y_1}{2}\right)
\delta\!\left(x_2-\frac{y_2}{2}\right),
$
which corresponds to symmetric binary fragmentation, where each parent particle splits into two equal fragments. Each particle fragments independently of the others. The simulations are carried out over the time interval $t \in [0,3]$. The numerical results are summarized in Table~\ref{tab:m4_m00} - Table~\ref{tab:m4_m10m01},  and the solution behavior for Case 4 is illustrated in Figure~\ref{fig:case4}. Mass conservation is maintained throughout the simulation. 
\begin{table}[H]
\centering
\caption{Comparison of exact and numerical values of $m_{00}(t)=1+2t$ for Test Case 4.}
\label{tab:m4_m00} 
\renewcommand{\arraystretch}{1.35}
\setlength{\tabcolsep}{8pt}

\begin{tabular}{c c cc cc cc}
\toprule
$t$ & Exact
& \multicolumn{2}{c}{$20\times20$}
& \multicolumn{2}{c}{$40\times40$}
& \multicolumn{2}{c}{$60\times60$} \\
\cmidrule(lr){3-4} \cmidrule(lr){5-6} \cmidrule(lr){7-8}
& 
& Num & Rel. Err.
& Num & Rel. Err.
& Num & Rel. Err. \\
\midrule

1.0 & 3
& 2.31496 & 1.48e-01
& 2.99766 & 7.79e-04
& 2.63283 & 3.14e-02 \\

1.5 & 4
& 3.01977 & 3.26e-01
& 3.98930 & 2.67e-03
& 4.03481 & 8.70e-03 \\

2.0 & 5
& 3.58432 & 5.15e-01
& 4.97188 & 5.62e-03
& 5.06814 & 1.36e-02 \\

2.5 & 6
& 3.95469 & 6.75e-01
& 5.94303 & 9.49e-03
& 6.11074 & 1.85e-02 \\

3.0 & 7
& 5.10811 & 2.70e-01
& 6.90046 & 1.42e-02
& 7.16012 & 2.29e-02 \\

\bottomrule
\end{tabular}
\end{table}

\begin{table}[H]
\centering
\caption{Comparison of exact and numerical values of $m_{10}(t)+m_{01}(t)=2$ for Test Case 4.}
\label{tab:m4_m10m01}
\renewcommand{\arraystretch}{1.35}
\setlength{\tabcolsep}{8pt}

\begin{tabular}{c c cc cc cc}
\toprule
$t$ & Exact
& \multicolumn{2}{c}{$20\times20$}
& \multicolumn{2}{c}{$40\times40$}
& \multicolumn{2}{c}{$60\times60$} \\
\cmidrule(lr){3-4} \cmidrule(lr){5-6} \cmidrule(lr){7-8}
& 
& Num & Rel. Err.
& Num & Rel. Err.
& Num & Rel. Err. \\
\midrule

1.0 & 2
& 1.93237 & 3.38e-02
& 2.01932 & 9.66e-03
& 1.99982 & 9.18e-05 \\

1.5 & 2
& 1.84533 & 7.73e-02
& 2.03176 & 1.59e-02
& 2.03957 & 1.98e-02 \\

2.0 & 2
& 1.72423 & 1.38e-01
& 2.04287 & 2.14e-02
& 2.05826 & 2.91e-02 \\

2.5 & 2
& 1.57697 & 2.12e-01
& 2.05265 & 2.63e-02
& 2.07755 & 3.88e-02 \\

3.0 & 2
& 1.68453 & 1.58e-01
& 2.06151 & 3.08e-02
& 2.09716 & 4.86e-02 \\

\bottomrule
\end{tabular}
\end{table}

%
%
%

\textbf{Test Case 5: }In this example, we consider the three–dimensional linear fragmentation equation with  breakage kernel $\beta(\mathbf{x},\mathbf{y})= \frac{8}{y_1y_2y_3}$
and a  selection function
$\Gamma(x_1, x_2) = x_1x_2x_3$.
The initial condition is prescribed as
$u(x_1, x_2,x_3, 0) = \delta(x_1-1)\delta(x_2-1)\delta(x_3-1).$
The exact moments are $m_{00}(t)=1+7t$ and $m_{111}(t)=1$.
Results are given in Table~\ref{tab:m5_m00}- Table~\ref{tab:m5_m111}
\begin{table}[H]
\centering
\caption{\textbf{Comparison of exact and numerical values of $m_{00}(t)=1+7t$ for Test Case 5.}}
\label{tab:m5_m00} 

\small
\renewcommand{\arraystretch}{1.3}
\setlength{\tabcolsep}{8pt}

\begin{tabular}{c|c|cc|cc|cc}
\toprule
$t$ & Exact 
& \multicolumn{2}{c|}{$5\times5$}
& \multicolumn{2}{c|}{$10\times10$}
& \multicolumn{2}{c}{$15\times15$} \\

\cmidrule(r){3-4}
\cmidrule(r){5-6}
\cmidrule(l){7-8}

& 
& Num & Rel.Err
& Num & Rel.Err
& Num & Rel.Err \\
\midrule

1
& 8
& 7.99124 & 0.001095
& 8.60212 & 0.075265 
& 8.28453 & 0.0355658 \\

1.5
& 11.5 
& 12.9037 & 0.122060
& 12.7629 & 0.109820 
& 12.1138 & 0.053377 \\

2.0
& 15
& 18.8128 & 0.254186
& 17.1758 & 0.145055
& 16.0794 & 0.071959 \\

2.5 
& 18.5 
& 25.8339 & 0.396428 
& 21.8478 & 0.180961
& 20.1857 & 0.091119 \\

3.0
& 22 
& 34.0938 & 0.549718
& 26.7855 & 0.217523
& 24.4361 & 0.110733 \\

\bottomrule
\end{tabular}
\end{table}

\begin{table}[H]
\centering
\caption{\textbf{Comparison of exact and numerical values of $m_{111}(t)=1$ for Test Case 5.}}
\label{tab:m5_m111}

\small
\renewcommand{\arraystretch}{1.3}
\setlength{\tabcolsep}{8pt}

\begin{tabular}{c|c|cc|cc|cc}
\toprule
$t$ & Exact 
& \multicolumn{2}{c|}{$5\times5$}
& \multicolumn{2}{c|}{$10\times10$}
& \multicolumn{2}{c}{$15\times15$} \\

\cmidrule(r){3-4}
\cmidrule(r){5-6}
\cmidrule(l){7-8}

& 
& Num & Rel.Err
& Num & Rel.Err
& Num & Rel.Err \\
\midrule

1
& 1
& 1.00214 & 0.002141
& 1.08075 & 0.080748
& 1.0411 & 0.041105 \\

1.5
& 1
& 1.12654 & 0.126542
& 1.11830 & 0.118303
& 1.06233 & 0.062326 \\

2.0
& 1
& 1.26176 & 0.261756
& 1.15636 & 0.156356
& 1.08408 & 0.084081 \\

2.5 
& 1
& 1.40886 & 0.408862
& 1.19480 & 0.194797
& 1.10607 & 0.106075 \\

3.0
& 1
& 1.56900 & 0.568997
& 1.23364 & 0.233641
& 1.12821 & 0.128207 \\

\bottomrule
\end{tabular}
\end{table}
%
%
\subsection{Rate of convergence test}

\textbf{Test Case 1: }In this example, we consider the two–dimensional linear fragmentation equation with the binary breakage kernel $
\beta(x_1,x_2 \mid y_1,y_2)=\frac{2}{y_1 y_2},
$
together with the constant selection function
$
\Gamma(x_1,x_2)=1.
$
For this problem, the analytical solution is given by
\[
u(x,y,t)
=
\left(1 + t\right)^3 \exp\left(-(x + y)(1 + t)\right)
\]
The analysis is performed over the computational domain $
\mathcal{D} := [10^{-9},\,2] \times [10^{-9},\,2]
$. To examine the convergence behavior, the grid is successively refined by increasing the number of grid points in each spatial direction at every iteration. 
The numerical results for uniform and non-uniform meshes are summarized in Table~\ref{tab:m6} and Table~\ref{tab:m61}. The findings are consistent with the theoretical analysis presented 

\begin{table}[htbp]
\centering
\caption{Error table for the PBE involving the breakage-growth process in uniform mesh for test case 1}
\label{tab:m6}
\renewcommand{\arraystretch}{1.2}
\begin{tabular}{c c c c c c c}
\hline
$h$ & $\|\mathcal{E}\|_{L^2(\mathcal{D} )}$ & EOC 
    & $\|\mathcal{E}\|_{H^1(\mathcal{D} )}$ & EOC
    & $\mathrm{RelError}_{L^2}$ & Order \\
\hline
\multicolumn{7}{c}{\textbf{$P_1$ polynomial space}} \\
\hline

0.707107    & 0.033277 & -- & 0.222819 & -- & 0.067782 & -- \\
0.353553   & 0.00844127 & 1.97899  & 0.111903 & 0.99362 &  0.017194 & 1.979 \\
0.176777    & 0.00212476  & 1.99016 & 0.056013& 0.998416 & 0.00432791 & 1.99016  \\
0.0883883   & 0.000533906 & 1.99264 & 0.0280139 & 0.999617 &  0.00108751 & 1.99264  \\
0.0441942   & 0.000134137  & 1.99288 & 0.0140079 & 0.999909 & 0.000273222 & 1.99288 \\
\hline

\multicolumn{7}{c}{\textbf{$P_2$ polynomial space}} \\
\hline
0.707107    & 0.00130762  & -- & 0.0227965 & -- &0.0026635& -- \\
0.353553   & 0.000165226 & 2.98443   & 0.00579873 & 1.975 &  0.000336548 & 2.98443 \\
0.176777    & $2.07149e-05 $  & 2.9957 & 0.00145603  &  1.99369 & $4.2194e-05$ & 2.9957  \\
0.0883883   & $2.59181e-06  $ & 2.99864 & 0.000364406 & 1.99843&  $ 5.27924e-06$ & 2.99864  \\
0.0441942   & $3.24087e-07$  & 2.99951  & $9.11263e-05$ & 1.99961 & $6.60131e-07 $&  2.99951 \\
\hline
\multicolumn{7}{c}{\textbf{$P_3$ polynomial space}} \\
\hline
0.707107    & $9.59004e-05$  & -- & 0.00132486 & -- &0.00019534& -- \\
0.353553   & $6.11619e-06$ & 3.97083   & 0.000168326 & 2.9765 &  $1.2458e-05$ & 3.97084 \\
0.176777    & $3.80829e-07 $  & 4.00542 & $2.11135e-05$ & 2.99502 & $7.75708e-07$ & 4.00542  \\
0.0883883   & $2.29431e-08 $ & 4.05301 & $2.63903e-06$ & 3.00009 &  $4.67327e-08$ & 4.05301  \\
0.0441942   & $1.73022e-09$  & 3.72904 & $3.31171e-07$ & 2.99436 & $3.52426e-09 $& 3.72904  \\
\hline
\end{tabular}
\end{table}

\begin{table}[htbp]
\centering
\caption{Error table for the PBE involving the breakage-growth process in  non uniform mesh for test case 1}
\label{tab:m61}
\renewcommand{\arraystretch}{1.2}
\begin{tabular}{c c c c c c c}
\hline
$h$ & $\|\mathcal{E}\|_{L^2(\mathcal{D} )}$ & EOC 
    & $\|\mathcal{E}\|_{H^1(\mathcal{D} )}$ & EOC
    & $\mathrm{RelError}_{L^2}$ & Order \\
\hline
\multicolumn{7}{c}{\textbf{$P_1$ polynomial space}} \\
\hline

2.12132     & 0.133092 & -- & 0.404926 & -- & 0.271184  & -- \\
1.23744   & 0.0346368  & 2.49747  & 0.206 & 1.25387 &  0.0705519 & 2.49807  \\
0.864242    & 0.0155253  & 2.23557  & 0.137661 & 1.12296 & 0.0316234 & 2.23558   \\
0.662913   & 0.00876396 &  2.15613  &  0.103329  & 1.08169 & 0.0178512 & 2.15613  \\
0.537401  & 0.00562001  & 2.11681 & 0.082694 & 1.06134 & 0.0114473   & 2.11681 \\
\hline

\multicolumn{7}{c}{\textbf{$P_2$ polynomial space}} \\
\hline
2.12132     &  0.0141101 & -- & 0.0869338 & -- & 0.0287503 & -- \\
1.23744   & 0.00179424 & 3.82621  & 0.0233369 & 2.43993 & 0.00365469 & 3.82681 \\
0.864242    &0.000534 & 3.37639 & 0.0105086  & 2.22274 & 0.0010877  & 3.37641    \\
0.662913   & 0.000225624  & 3.24847  &  0.00593831& 2.15214 & 0.000459573 & 3.24847   \\
0.537401  & 0.0001156  & 3.18602  & 0.00380863 & 2.11605 & 0.000235464  & 3.18602 \\
\hline
\multicolumn{7}{c}{\textbf{$P_3$ polynomial space}} \\
\hline

2.12132     & 0.00250041 & -- & 0.0130814 & -- & 0.00509474 & -- \\
1.23744   & 0.00016762 & 5.01396    &  0.00171656 & 3.76787 & 0.000341426  & 5.01457   \\
0.864242    & $3.35112e-05$ & 4.48486  & 0.000513733 & 3.3609 & $6.82587e-05$ & 4.48488   \\
0.662913   & $1.06163e-05$  & 4.33427  & 0.000217476  & 3.24128 & $2.16243e-05$ & 4.33427    \\
0.537401  & $4.33462e-06$ &  4.26758   & 0.000111536 & 3.18123 & $8.82915e-06$ & 4.26758 \\
\hline
\end{tabular}
\end{table}

\textbf{Test Case 2: }
In this example, we consider the two–dimensional linear fragmentation equation with the multiple breakage kernel $
\beta(x_1,x_2 \mid y_1,y_2)=\frac{4}{y_1 y_2},
$
together with the constant selection function
$
\Gamma(x_1,x_2)=1.
$
For this problem, the analytical solution is given by
\[
u(x,y,t)
=
\frac{(1+t)^4}
{\left(1+(1+t)x\right)^3
 \left(1+(1+t)y\right)^3}
\]
The computational domain is defined as $
\mathcal{D} := [10^{-9},\,2] \times [10^{-9},\,2]$ .To examine the convergence behavior, we perform a grid refinement study by increasing the number of nodes in each spatial dimension. The numerical results are summarized in Tables \ref{tab:m7} and \ref{tab:m71} for uniform and non-uniform meshes, respectively. These findings provide empirical confirmation of the theoretical convergence analysis established in Section~\ref{sec:4}.

	\begin{table}[htbp]
\centering
\caption{Error table for the PBE involving the breakage-growth process in uniform mesh for test case 2}
\label{tab:m7}
\renewcommand{\arraystretch}{1.2}
\begin{tabular}{c c c c c c c}
\hline
$h$ & $\|\mathcal{E}\|_{L^2(\mathcal{D} )}$ & EOC 
    & $\|\mathcal{E}\|_{H^1(\mathcal{D} )}$ & EOC
    & $\mathrm{RelError}_{L^2}$ & Order \\
\hline
\multicolumn{7}{c}{\textbf{$P_1$ polynomial space}} \\
\hline

0.707107    & 0.0907978 & -- & 0.565153 & -- & 0.456477 & -- \\
0.353553   & 0.0258793 & 1.81086  & 0.320246 & 0.81946 &  0.129898 & 1.81316 \\
0.176777    & 0.00681066 & 1.92593 & 0.166723 & 0.941733& 0.0341839 & 1.92599   \\
0.0883883   & 0.00174589 & 1.96383  & 0.0842775 & 0.98423 &  0.00876296 & 1.96383  \\
0.0441942   & 0.000444965  & 1.97178 & 0.0422548 & 0.996033 & 0.00223402 & 1.97178  \\
\hline

\multicolumn{7}{c}{\textbf{$P_2$ polynomial space}} \\
\hline
0.707107    & 0.0118579   & -- & 0.177086 & -- &0.0596322 & -- \\
0.353553   & 0.00176825& 2.74546   & 0.0566981& 1.975 &  0.00887816 &  2.74776  \\
0.176777    &0.000234843  & 2.91255 & 0.0153886   &  1.99369 & 0.00117907& 2.91261 \\
0.0883883   & $2.99464e-05  $ & 2.97125 & 0.00393543 & 1.99843&  0.00015035 &  2.97125   \\
0.0441942   & $3.77567e-06 $  &2.98758 & 0.000989574 & 1.99961 & 1.89563e-05 &  2.98758 \\
\hline
\multicolumn{7}{c}{\textbf{$P_3$ polynomial space}} \\
\hline
0.707107    &  0.00285561  & -- & 0.0380027 & -- &0.0143565& -- \\
0.353553   & 0.000246738 & 3.53234   & 0.00644771 & 2.55924 & 0.00123884 & 3.53464 \\
0.176777    & $1.7301e-05$  & 3.83405 & 0.000894784 & 2.84918 & $8.68622e-05$ & 3.83412  \\
0.0883883   & $1.12455e-06$ & 3.94343 & 0.000115144 & 2.95811 &  $5.64597e-06$ & 3.94343  \\
0.0441942   & $7.14667e-08$  & 3.97593 & $1.44945e-05$ & 2.98986 & $ 3.58809e-07 $& 3.97593  \\

\hline
\end{tabular}
\end{table}	

\begin{table}[htbp]
\centering
\caption{Error table for the PBE involving the breakage-growth process in non uniform mesh for test case 2}
\label{tab:m71}
\renewcommand{\arraystretch}{1.2}
\begin{tabular}{c c c c c c c}
\hline
$h$ & $\|\mathcal{E}\|_{L^2(\mathcal{D} )}$ & EOC 
    & $\|\mathcal{E}\|_{H^1(\mathcal{D} )}$ & EOC
    & $\mathrm{RelError}_{L^2}$ & Order \\
\hline
\multicolumn{7}{c}{\textbf{$P_1$ polynomial space}} \\
\hline

2.12132     & 0.124822 & -- & 0.620504  & -- & 0.629382 & -- \\
1.23744   & 0.03686  & 2.26302   & 0.334221 & 1.14793 & 0.18509  & 2.2707   \\
0.662913   & 0.0095572  & 2.16266   &  0.175057  & 1.03611 &0.0479835 &  2.1629    \\
0.342505   & 0.0024189  & 2.08066  & 0.0884951 & 1.03302 & 0.0121445 & 2.08067     \\
0.174015   & 0.000608845 &  2.03723   & 0.0443704   & 1.01953 &0.0030568 & 2.03723 \\
\hline

\multicolumn{7}{c}{\textbf{$P_2$ polynomial space}} \\
\hline
2.12132     &  0.0216072 & -- & 0.212275  & -- &0.108949 & -- \\
1.23744   & 0.00351271 & 3.3704   &  0.0683196 & 2.10332 & 0.0176388  & 3.37808   \\
0.662913   & 0.000443517 & 3.31554  & 0.0185764  & 2.08651 & 0.00222675& 3.31578    \\
0.342505   & $5.56052e-05$  & 3.14445  & 0.00473578  & 2.06971 & 0.000279175& 3.14445     \\
0.174015   & $6.95541e-06$ &  3.06988   &  0.00118974  & 2.04005 & $3.49207e-05$ & 3.06988\\

\hline
\multicolumn{7}{c}{\textbf{$P_3$ polynomial space}} \\
\hline
2.12132     & 0.00655992 & -- & 0.0545344  & -- & 0.0330767& -- \\
1.23744   & 0.000637708 & 4.32443   &  0.0100451& 3.1387 & 0.0032022  & 4.33211   \\
0.662913   & $4.2035e-05$ & 4.3569  & 0.00134639 & 3.21981 & 0.000211044 & 4.35714   \\
0.342505   & $2.66185e-06$  & 4.17877  & 0.000171332  & 3.12192 & $1.33642e-05$ & 4.17877    \\
0.174015   & $1.65587e-07 $ &  4.10145   & $2.15119e-05$ & 3.06432 & $8.31353e-07 $ & 4.10145 \\

\hline
\end{tabular}
\end{table}	


\textbf{Test Case 3: }
Consider the three–dimensional linear fragmentation equation with the binary breakage kernel $
\beta(x_1,x_2,x_3 \mid y_1,y_2,y_3)=\frac{2}{y_1 y_2y_3},
$
together with the constant selection function
$
\Gamma(x_1,x_2,x_3)=1.
$
For this problem, the analytical solution is given by
\[
u(x,y,t)
=
\left(1 + t\right)^4 \exp\left(- (x + y + z)\left(1 + t\right)\right)
\]
The computational domain is defined as $
\mathcal{D} := [10^{-9},\,2] \times [10^{-9},\,2]\times [10^{-9},\,2],
$
and a uniform time step size is employed. To examine the convergence behavior, the grid is successively refined by increasing the number of grid points in each spatial direction at every iteration. 
The numerical results, summarized in Table~\ref{tab:m8} and Table~\ref{tab:m81}. The findings are consistent with the theoretical analysis presented.
	
	\begin{table}[htbp]
\centering
\caption{Error table for the PBE involving the breakage-growth process in uniform mesh for test case 3}
\label{tab:m8}
\renewcommand{\arraystretch}{1.2}
\begin{tabular}{c c c c c c c}
\hline
$h$ & $\|\mathcal{E}\|_{L^2(\mathcal{D} )}$ & EOC 
    & $\|\mathcal{E}\|_{H^1(\mathcal{D} )}$ & EOC
    & $\mathrm{RelError}_{L^2}$ & Order \\
\hline
\multicolumn{7}{c}{\textbf{$P_1$ polynomial space}} \\
\hline
1.73205 & 0.182653  & -- & 0.600829  & -- & 0.531406 & -- \\
0.866025 & 0.0472836 & 1.94969  & 0.290166 & 1.05008 &  0.137465 & 1.95075 \\
0.57735 & 0.0214541  & 1.94899 & 0.191897 & 1.0198 & 0.0623713 &  1.94902  \\
0.433013 & 0.0122741& 1.94113   & 0.143485 & 1.01059 & 0.0356831 & 1.94113 \\
0.34641 & 0.00797185  & 1.93405 & 0.114618 & 1.00664 & 0.0231757 & 1.93405  \\
\hline

\multicolumn{7}{c}{\textbf{$P_2$ polynomial space}} \\
\hline
1.73205    & 0.0185723   & -- & 0.150071 & -- & 0.0540338 & -- \\
0.866025   & 0.00239078 & 2.9576   & 0.0402155 & 1.89982 &  0.00695055 &  2.95866  \\
0.57735   & 0.000713045   & 2.9838  & 0.0181186  & 1.96642 & 0.00207296 & 2.98384 \\
0.433013   & 0.000302002 & 2.98632 & 0.0102402 & 1.9835 &  0.000877978&  2.98632  \\
0.34641    & 0.00015513  & 2.98538 & 0.00656792 & 1.9903 & 0.000450993 &  2.98538  \\
\hline
\multicolumn{7}{c}{\textbf{$P_3$ polynomial space}} \\
\hline
1.73205    & 0.00303667  & -- & 0.0271831 & -- & 0.00883479& -- \\
0.866025   & 0.000211616 & 3.84296   & 0.00370841 & 2.87384 & 0.000615219 &  3.84402  \\
0.57735   & $4.30523e-05$  & 3.92724  &  0.00113441  & 2.92131 & 0.000125161 & 3.92728 \\
0.433013   & $1.38192e-05$ & 3.95004 & 0.000523708 & 2.68676 &  $4.01751e-05$ &  3.95004   \\
0.34641    & $5.74964e-06 $ & 3.92986 &  0.000345783 & 1.86034 & $1.67153e-05$ &  3.92986 \\

\hline
\end{tabular}
\end{table}		

\begin{table}[htbp]
\centering
\caption{Error table for the PBE involving the breakage-growth process in non uniform mesh for test case 3}
\label{tab:m81}
\renewcommand{\arraystretch}{1.2}
\begin{tabular}{c c c c c c c}
\hline
$h$ & $\|\mathcal{E}\|_{L^2(\mathcal{D} )}$ & EOC 
    & $\|\mathcal{E}\|_{H^1(\mathcal{D} )}$ & EOC
    & $\mathrm{RelError}_{L^2}$ & Order \\
\hline
\multicolumn{7}{c}{\textbf{$P_1$ polynomial space}} \\
\hline

3.4641 & 0.67955  & -- & 1.29624& -- & 2.02379 & -- \\
1.9245 & 0.0800848 & 3.63796  & 0.354238 & 2.20702 & 0.232842 & 3.67883 \\
1.51554 & 0.045787  & 2.34033 & 0.263051  & 1.24584 & 0.133114  &  2.34062  \\
1.05848  & 0.0207684  & 2.20248  & 0.174111   & 1.14963 & 0.0603778  & 2.20252 \\
0.919047 & 0.0153536  & 2.13867& 0.149025 & 1.10147 & 0.0446362 & 2.13864  \\
\hline

\multicolumn{7}{c}{\textbf{$P_2$ polynomial space}} \\
\hline

3.4641 & 0.131376  & -- & 0.486425 & -- & 0.391255 & -- \\
1.9245 & 0.00661875 & 5.08374 & 0.0659134 & 3.40045  & 0.0192436 & 5.12461  \\
1.51554 & 0.00280493  & 3.5938 & 0.0378203  & 2.3253 & 0.0081546  &  3.59409   \\
1.05848  & 0.000833454 & 3.38089  & 0.0170509 & 2.2194  & 0.00242301 & 3.38093\\
0.919047 & 0.000525118  & 3.27054  & 0.0125653  & 2.1612 & 0.00152663  & 3.27051  \\
\hline
\multicolumn{7}{c}{\textbf{$P_3$ polynomial space}} \\
\hline
3.4641 & 0.0346772  & -- & 0.163219 & -- & 0.103273 & -- \\
1.9245 & 0.000876107 & 6.25797  & 0.00932043 & 4.87062 & 0.00254723 & 6.29883 \\
1.51554 & 0.000283647  & 4.72078 & 0.0040198  & 3.52032 & 0.000824631 &  4.72108   \\
1.05848  & $ 5.6599e-05$ & 4.49022  & 0.00122292  & 3.31524 & 0.000164544 & 4.49026 \\
0.919047 & $3.04266e-05$  & 4.3942 & 0.000792226 & 3.07371 & $8.84565e-05$ & 4.3942  \\

\hline
\end{tabular}
\end{table}		
\FloatBarrier
\section{Conclusion}\label{sec:6}
This work establishes a robust FEM-BDF2 method for solving multidimensional PBEs. To evaluate the precision and efficiency of the proposed method, it was applied to eight different combinations of selection functions, breakage kernels, and initial conditions for 2D and 3D fragmentation PBEs, using various grid sizes and basis functions. We have observed that this method is not only mathematically stable but also maintains high accuracy for moment preservation for both 2D and 3D regimes. Key features of the scheme include its robustness, validated by stability analysis, which verifies that the BDF2 scheme yields bounded solutions; its applicability to any grid; and its high accuracy in predicting both zeroth- and first-order moments. These findings suggest that the BDF2 scheme is very useful where both accuracy and computational speed are paramount. In future work, our aim is to extend the algorithm to accommodate more complex mechanisms, such as nonlinear breakage effects arising from collisional fragmentation.

\noindent {\bf Data availability:}
Data are available from the authors upon request.

\noindent {\bf Declaration of competing interest:} 
The authors declare no known competing financial interests or personal relationships that could have influenced the work reported in this paper.

\noindent {\bf  Acknowledgments:}  The first author gratefully acknowledges the Ministry of Education for supporting her research.

\bibliographystyle{abbrv}
\bibliography{ref}
\end{document}